\documentclass[11pt]{article}
\usepackage{layout}
\usepackage[makeroom]{cancel}
\usepackage{bbm}

\usepackage{amsfonts,dsfont}
\usepackage{amsmath}
\usepackage{amssymb,bm}
\usepackage{amsthm}
\usepackage{graphicx} 
\usepackage{mathtools}
\usepackage{multicol}
\usepackage{mathrsfs}
\usepackage{enumerate,enumitem}
\usepackage[rgb,dvipsnames]{xcolor}
\usepackage{float}
\usepackage[normalem]{ulem}
\usepackage{circuitikz}
\usepackage{cite}
\usepackage[colorlinks,linkcolor=blue,citecolor=blue,pagebackref,hypertexnames=false, breaklinks]{hyperref}
\usepackage{float}

\newtheorem{theorem}{Theorem}[section]

\newtheorem{definition}[theorem]{Definition}
\newtheorem{proposition}[theorem]{Proposition}

\newtheorem{corollary}[theorem]{Corollary}
\newtheorem{lemma}[theorem]{Lemma}
\newtheorem{remark}[theorem]{Remark}
\newtheorem{example}[theorem]{Example}
\newtheorem{examples}[theorem]{Examples}

\newtheorem{conjecture}[theorem]{Conjecture}
\newtheorem{open question}[theorem]{Open Question}
\newtheorem{c/p}[theorem]{Conjecture/Proposition}

\def\vint{\mathop{\mathchoice%
 {\setbox0\hbox{$\displaystyle\intop$}\kern 0.22\wd0%
 \vcenter{\hrule width 0.6\wd0}\kern -0.82\wd0}%
 {\setbox0\hbox{$\textstyle\intop$}\kern 0.2\wd0%
 \vcenter{\hrule width 0.6\wd0}\kern -0.8\wd0}%
 {\setbox0\hbox{$\scriptstyle\intop$}\kern 0.2\wd0%
 \vcenter{\hrule width 0.6\wd0}\kern -0.8\wd0}%
 {\setbox0\hbox{$\scriptscriptstyle\intop$}\kern 0.2\wd0%
 \vcenter{\hrule width 0.6\wd0}\kern -0.8\wd0}}%
 \mathopen{}\int}

\newcommand{\R}{\mathbb R}

\mathtoolsset{showonlyrefs=true}

\usepackage[textwidth=3.2cm]{todonotes}

\title{Gagliardo-Nirenberg, Trudinger-Moser and Morrey inequalities  on Dirichlet spaces}

\author{Patricia Alonso Ruiz\footnote{P.A.R. was partly supported by the NSF grant DMS~1951577.}, Fabrice Baudoin\footnote{F.B. was partly supported by the NSF grant DMS~1901315.}}

\date{\today}

\begin{document}

\maketitle

\begin{abstract}
With a view towards Riemannian or sub-Riemannian manifolds, RCD metric spaces and specially fractals, 
this paper makes a step further in the development of
a 
theory of heat semigroup based $(1,p)$ Sobolev spaces 
in the general framework of Dirichlet spaces. 
Under suitable assumptions that are verified in a variety of settings, 
the tools developed by D. Bakry, T. Coulhon, M. Ledoux and L. Saloff-Coste in the paper~\textit{Sobolev inequalities in disguise} 
allow us to obtain
the whole family of Gagliardo-Nirenberg and Trudinger-Moser inequalities with optimal exponents. 
The latter 
depend not only on the Hausdorff and walk dimensions of the space but also on other invariants. 
In addition, we prove Morrey type inequalities and apply them to study the infimum of the exponents that ensure continuity of Sobolev functions. The results are illustrated for fractals using the Vicsek set, whereas several conjectures are made for nested fractals and the Sierpinski carpet.
\end{abstract}

\tableofcontents

\section{Introduction}


The theory of Sobolev spaces was first pushed forward in order to prove solvability of certain partial differential equations, see for example~\cite{Maz11}. When $X$ is a Riemannian manifold, a function 
$f\in L^p(X)$ is said to be in the Sobolev space $W^{1,p}(X)$ if its distributional gradient is given by a 
vector-valued function $\nabla f\in L^p(X:\R^n)$. 
In more general spaces, a distributional theory of derivatives relying on integration by parts may not be available, which makes necessary to find an alternative notion of derivative.

\medskip

After the seminal paper of J. Cheeger \cite{Chee99}, many authors introduced 
in different ways 
a variety of notions of a gradient in the general context of metric measure spaces; 
we refer for instance to 
the book by J. Heinonen~\cite{Hei01} and the references therein. Those gradients naturally yield a rich theory of first order Sobolev spaces that was developed around stepstone works like the ones by N. Shanmugalingam~\cite{Sha00}; see also the book~\cite{HKST15} and the more recent papers by L. Ambrosio, M. Colombo and S. Di Marino~\cite{ACD15}, and G. Savar\'e~\cite{Sav19}.

\medskip

The approach to Sobolev spaces undertaken in the above cited references crucially relies on a notion of a measure-theoretic gradient that requires the underlying space to admit enough ``good'' rectifiable curves, a property that may not be present in some singular, fractal-like, metric measure spaces. With the aim of including these, potential-theoretic based definitions have been introduced and studied at different levels of generality, 
see e.g.~\cite{Str03,HZ03,PP10} and references therein.
The present paper is set up in the framework of Dirichlet spaces that are general enough to also cover this type of fractals.

\medskip

Dirichlet spaces are measure spaces equipped with a closed Markovian symmetric bilinear form $\mathcal{E}$, called Dirichlet form, whose domain is dense in $L^2$. Dirichlet spaces provide a unified framework to study doubling metric measure spaces supporting a 2-Poincar\'e inequality~\cite{KST04}, fractals~\cite{Kig01}, infinite-dimensional spaces~\cite{BH91} and non-local operators~\cite{CF12}. An important tool available in any Dirichlet space is the heat semigroup. 
The latter is a priori
an $L^2$ object, meaning that it is originally defined on $L^2$ 
by means of the Dirichlet form $\mathcal{E}$ itself 
using spectral theory of Hilbert spaces. However, the Markovian property of 
$\mathcal{E}$ 
and classical interpolation theory allow to define this semigroup as a family of operators acting on any $L^p$ space, $1 \le p \le +\infty$.

\medskip

The latter extension was used in~\cite{ABCRST1} to develop 
a theory of $L^p$ Besov type spaces that have systematically been studied in the context of strictly local spaces~\cite{ABCRST2}, strongly local spaces with sub-Gaussian heat kernel estimates~\cite{ABCRST3} and non-local spaces~\cite{ABCRST4}. While the papers~\cite{ABCRST2,ABCRST3} primarily dealt with the $L^1$ theory and the associated theory of bounded variation (BV) functions and sets of finite perimeter, the present paper focuses on the $L^p$ theory for $p>1$. 
The Sobolev spaces considered here arise as $L^p$ Besov spaces at the critical exponent, c.f. Definition~\ref{D:pSobolev_class}, and
coincide with their classical counterpart in the Riemannian and other often studied metric measure settings, see Section~\ref{sec: examples}. This heat semigroup approach digresses from existing generalizations of the classical ideas of Mazy'a~\cite{Maz11} to fractals, see e.g.~\cite{HRT13,HKM18}. 

\medskip

Once Sobolev spaces have been identified, it is natural to investigate analogues of the famous Gagliardo-Nirenberg and Trudinger-Moser inequalities. Such inequalities classically play an important role in the study of partial differential equations and include as special cases the Sobolev embedding inequality, the Nash inequality and the Ladyzhenskaya's inequality to name but a few. Besides their applications to partial differential equations, Gagliardo-Nirenberg and Trudinger-Moser inequalities also carry geometric information and, in the context of Riemannian geometry, they have for instance been applied to the study of sets of finite perimeter, conformal geometry~\cite{CY03} and cohomology~\cite{Pan08}. In the context of metric measure spaces, they have been closely related to the study of quasi-conformal or quasi-symmetric maps and invariants, see~\cite{HK98}.

\medskip

The paper is organized as follows: Section~\ref{sec2} introduces the Sobolev spaces $W^{1,p}(\mathcal E)$, $p\geq 1$, associated with a general Dirichlet form $\mathcal E$. These are characterized in Section~\ref{sec: examples} for various specific classes of examples. In strictly local Dirichlet spaces, which admit a canonical gradient structure intrinsically associated to the form, it is showed in Theorem~\ref{T:BV2} that, under suitable conditions, $W^{1,p}(\mathcal E)$ coincides with the Sobolev space defined by that gradient structure. In the case of strongly local Dirichlet spaces, which includes many fractals, $W^{1,p}(\mathcal E)$ is characterized in Theorem~\ref{T:BV3} as a Korevaar-Schoen space.
Section~\ref{section GN-TM} is devoted to the study of Gagliardo-Nirenberg and Trudinger-Moser inequalities in general Dirichlet spaces, c.f.\ Theorem~\ref{T:Gagliardo-Nirenberg} and Corollary~\ref{C:Trudinger-Moser_1}. The techniques rely on the general methods proposed by D. Bakry, T. Coulhon, M. Ledoux and L. Saloff-Coste in the paper~\cite{BCLS95};
besides the ultracontractivity of the semigroup, the main assumption is an $L^p$ pseudo-Poincar\'e inequality that is related to a weak notion of curvature (in the Bakry-\'Emery sense) of the underlying space. The latter is shown to be satisfied in large classes of examples like RCD spaces or nested fractals. 
Finally, Section~\ref{morrey} investigates embedding of the Sobolev spaces into spaces of H\"older functions. Of particular interest is the infimum $\delta_\mathcal{E}$ of the exponents for which such embedding occurs. In strictly local spaces and under suitable assumptions it is possible to bound above this quantity by the Hausdorff dimension of the space, c.f. Theorem~\ref{T:continuity_BV2}. In the case of fractals, Theorem~\ref{T:Vicsek continuity} shows that for the Vicsek set $\delta_{\mathcal{E}}=1$. Moreover, it is conjectured that for the Sierpinski gasket also $\delta_{\mathcal{E}}=1$, whereas for the Sierpinski carpet
\[
\delta_\mathcal{E}=1+\frac{\log 2}{d_W\log 3-2\log 2},
\]
where $d_W\approx 2.097$ is the so-called walk dimension of the carpet.


\subsection*{Notations} 

If $\Lambda_1$ and $\Lambda_2$ are functionals defined on a class of functions $f \in \mathcal C$, the notation
\[
\Lambda_1 (f) \simeq \Lambda_2 (f)
\]
means that there exist constants $c,C >0$ such that for every $f \in \mathcal C$
\[
c\Lambda_1 (f) \le \Lambda_2 (f)\le C \Lambda_1 (f).
\]
Also, in proofs, $c,C$ will generically denote positive constants whose values may change from one line to another.

\section{Framework, basic definitions and preliminaries}\label{sec2}


Throughout the paper, $X$ will denote a good measurable space (like a Polish or Radon space) equipped with a $\sigma$-finite measure $\mu$ supported on $X$. In addition, the pair $(\mathcal{E},\mathcal{F})$, where $\mathcal{F}={\rm dom}\,\mathcal{E}$, will denote a Dirichlet form on $L^2(X,\mu)$. We refer to $(X,\mu,\mathcal{E},\mathcal{F})$ as a Dirichlet space. Its associated heat semigroup $\{P_{t}\}_{t\geq 0}$ is always assumed to be conservative, i.e. $P_t1=1$. Further details about this setting can be found in~\cite{ABCRST1}.

\subsection{Heat semigroup-based BV, Sobolev and Besov classes}

Following~\cite{ABCRST1}, we define the (heat semigroup-based) Besov classes associated with a Dirichlet space$(X,\mu,\mathcal{E},\mathcal{F})$.

\begin{definition}
For any $p \ge 1$ and $\alpha \ge 0$, define
\begin{equation*}
\mathbf{B}^{p,\alpha}(X):=\left\{ f \in L^p(X,\mu)\, :\,  \limsup_{t  \to 0^+} t^{-\alpha} \left( \int_X P_t (|f-f(y)|^p)(y) d\mu(y) \right)^{1/p}<+\infty \right\}.
\end{equation*}
\end{definition}

The basic properties of the space $\mathbf{B}^{p,\alpha}(X)$ endowed with the semi-norm
\[
\| f \|_{p,\alpha}= \sup_{t >0} t^{-\alpha} \left( \int_X P_t (|f-f(y)|^p)(y) d\mu(y) \right)^{1/p}
\]
are studied in~\cite{ABCRST1}. 
In the present paper, we shall also be interested in the localized semi-norms defined for $R>0$ as
\[
\| f \|_{p,\alpha,R}:= \sup_{t  \in (0,R)} t^{-\alpha} \left( \int_X P_t (|f-f(y)|^p)(y) d\mu(y) \right)^{1/p}.
\]
Note that, in view of~\cite[Lemma 4.1]{ABCRST1}, one has for every $R>0$
\[
\| f \|_{p,\alpha,R} \le \| f \|_{p,\alpha} \le \frac{2}{R^\alpha} \| f \|_{L^p(X,\mu)} +\| f \|_{p,\alpha,R}
\]
and in particular all the norms $\| f \|_{L^p(X,\mu)} +\| f \|_{p,\alpha,R}$ are equivalent on $\mathbf{B}^{p,\alpha}(X)$ to the norm $\| f \|_{L^p(X,\mu)} +\| f \|_{p,\alpha}$. 

\medskip

The BV and Sobolev classes arise at the corresponding critical exponents as follows.

\begin{definition}
The class of heat semigroup based bounded variation (BV) functions is defined as 
\[
BV(\mathcal{E}):=\mathbf{B}^{1,\alpha_1}(X),
\]
where
\begin{equation*}
\alpha_1=\sup\{\alpha>0\,\colon\,\mathbf{B}^{1,\alpha}(X)\text{ contains non a.e. constant functions}\}.
\end{equation*}
For any $f \in BV(\mathcal{E})$, its total variation is defined as
\[
\mathbf{Var}_\mathcal{E} (f):=\liminf_{t  \to 0^+} t^{-\alpha_1}  \int_X P_t (|f-f(y)|)(y) d\mu(y).
\]
\end{definition}
As in the classical theory, the Sobolev classes are defined analogously for $p>1$.
\begin{definition}\label{D:pSobolev_class}
Let $p>1$. The $(1,p)$ heat semigroup based Sobolev class is defined as
\[
W^{1,p}(\mathcal{E}):=\mathbf{B}^{p,\alpha_p}(X),
\]
where
\begin{equation*}
\alpha_p:=\sup\{\alpha>0\,\colon\,\mathbf{B}^{p,\alpha}(X)\text{ contains non a.e. constant functions}\}.
\end{equation*}
For any $f \in W^{1,p}(\mathcal{E})$, its total $p$-variation is defined as
\[
\mathbf{Var}_{p,\mathcal{E}} (f):=\liminf_{t  \to 0^+} t^{-\alpha_p}  \left( \int_X P_t (|f-f(y)|^p)(y) d\mu(y) \right)^{1/p}.
\]
\end{definition}

\begin{remark}
For consistency in the notation, we will write $\mathbf{Var}_{1,\mathcal{E}} (f):=\mathbf{Var}_{\mathcal{E}} (f)$ for $f \in BV(\mathcal{E})$.
\end{remark}

\begin{remark}
From in~\cite[Proposition 4.6]{ABCRST1}, one has $\alpha_2=\frac{1}{2}$, $ W^{1,2}(\mathcal{E})=\mathbf{dom}\,\mathcal{E}=\mathcal{F}$ and $\mathbf{Var}_{2,\mathcal{E}} (f)=2\mathcal{E}(f,f)$.
\end{remark}

The following lemma shows that the functionals $\mathbf{Var}_{p,\mathcal{E}} (f)$ behave nicely with respect to cut-off arguments. This is a crucial property that will allow us to use the techniques developed by D. Bakry, T. Coulhon, M. Ledoux and L. Saloff-Coste in~\cite{BCLS95}.

\begin{lemma}\label{L:local_norm_approx:a}
For any nonnegative $f \in W^{1,p}(\mathcal E)$ if $p>1$, or $f\in BV(\mathcal E)$ if $p=1$, and any $\rho>0$, it holds that
\[
\left( \sum_{k \in \mathbb{Z}} \mathbf{Var}_{p,\mathcal{E}} (f_{\rho,k})^p \right)^{1/p} \le 2(p+1) \mathbf{Var}_{p,\mathcal{E}} (f),
\]
where $f_{\rho,k}:=(f-\rho^k)_+ \wedge \rho^k(\rho-1)$, $k \in \mathbb{Z}$. 
\end{lemma}

\begin{proof}
Let $p_t(y,dx)$ denote the heat kernel measure of the semigroup $P_t$, which exists because $(X,\mu)$ is assumed to be a good measurable space, c.f.~\cite[Theorem 1.2.3]{BGL14}). We first observe that, once we prove 
\begin{equation}\label{E:L:local_norm_approx:a_h1}
\sum_{k \in \mathbb{Z}} \int_X \int_X |f_{\rho,k} (x)-f_{\rho,k}(y)|^p p_t(y,dx)d\mu(y) \le 2(p+1) \int_X \int_X |f (x)-f(y)|^p p_t(y,dx)d\mu(y)
\end{equation}
for any $\rho>0$, then
\begin{multline*}
\liminf_{t \to 0^+} \left(\sum_{k \in \mathbb{Z}}  t^{-p\alpha_p } \!\!\!\int_X \int_X |f_{\rho,k} (x)-f_{\rho,k}(y)|^p p_t(y,dx)d\mu(y) \right) \\
\le  2(p+1)\liminf_{t \to 0^+} t^{-p\alpha_p } \!\!\!\int_X \int_X |f (x)-f(y)|^p p_t(y,dx)d\mu(y).
\end{multline*}
Using the superadditivity of the $\liminf$ one concludes
\begin{multline*}
\sum_{k \in \mathbb{Z}}  \liminf_{t \to 0^+} t^{-p\alpha_p } \int_X \int_X |f_{\rho,k} (x)-f_{\rho,k}(y)|^p p_t(y,dx)d\mu(y) \\
\le 2(p+1) \liminf_{t \to 0^+} t^{-p\alpha_p } \int_X \int_X |f (x)-f(y)|^p p_t(y,dx)d\mu(y).
\end{multline*}
The inequality~\eqref{E:L:local_norm_approx:a_h1} can implicitly be found in the proof of~\cite[Lemma 7.1]{BCLS95} with $a=p$. We include here the details to provide the explicit constant. For each $k\in\mathbb{Z}$, set $f_k:=f_{\rho,k}$ and define $B_k=\{x\in X~\,\colon\,\rho^k<f\leq \rho^{k+1}\}$. In this way, the external integral on the left hand side of~\eqref{E:L:local_norm_approx:a_h1} is decomposed it into an integral over $B_k$ and $B_k^c$. For the integrals over $B_k$, since the mapping $f\mapsto f_{k}$ is a contraction, 
it follows that
\begin{equation}\label{E:L:local_norm_approx:a_h2}
\sum_{k\in\mathbb{Z}}\int_{B_k} \int_X |f_{k} (x)-f_{k}(y)|^p p_t(y,dx)d\mu(y)\leq 
\int_X \int_X |f (x)-f(y)|^p p_t(y,dx)d\mu(y).
\end{equation}
To perform the integrals over $B_k^{c}$, we decompose them as
\begin{multline*}
\sum_{k\in\mathbb{Z}}\int_{B_k^c} \int_{B_k} |f_{k} (x)-f_{k}(y)|^p p_t(y,dx)d\mu(y)
+\sum_{k\in\mathbb{Z}}\int_{B^c_k} \int_{B_k^c}|f_{k} (x)-f_{k}(y)|^p p_t(y,dx)d\mu(y)\\
=:\sum_{k\in\mathbb{Z}}J_1(k)+\sum_{k\in\mathbb{Z}}J_2(k).
\end{multline*}
Again, the contraction property of $f\mapsto f_{k}$ yields
\begin{multline*}
\sum_{k\in\mathbb{Z}}J_1(k)\leq \sum_{k\in\mathbb{Z}} \int_X \int_{B_k}|f_{k} (x)-f_{k}(y)|^p p_t(y,dx)d\mu(y)\\
\leq\int_X \sum_{k\in\mathbb{Z}}  \int_{B_k}|f_{k} (x)-f_{k}(y)|^p p_t(y,dx)d\mu(y)\leq \int_X \int_X |f (x)-f(y)|^p p_t(y,dx)d\mu(y).
\end{multline*}
On the other hand, notice that for any $(x,y)\in B_k^c\times B_k^c$ we have $|f_k(x)-f_k(y)|\neq 0$ only if
\[
(x,y)\in \{f(x)\leq\rho^k<f(y)\rho^{-1}\}\cup\{f(y)\leq\rho^k<f(x)\rho^{-1}\}=:Z_k\cup Z_k^*.
\]
Also, $|f_k(x)-f_k(y)|=\rho^k(\rho-1)$ for $(x,y)\in Z_k\cup Z_k^*$. Thus,
\begin{align*}
\sum_{k\in\mathbb{Z}}J_2(k)&\leq =\sum_{k\in \mathbb{Z}}\int_X\int_X\big(\mathbf{1}_{Z_k}(x,y)+\mathbf{1}_{Z^*_k}(x,y)\big)|f_k(x)-f_k(y)|^p p_t(y,dx)d\mu(y)\\
&=\int_X\int_X\sum_{k\in\mathbb{Z}}\big(\mathbf{1}_{Z_k}(x,y)+\mathbf{1}_{Z^*_k}(x,y)\big)\rho^{kp}(\rho-1)^pp_t(y,dx)d\mu(y).
\end{align*}
One can now prove, see~\cite[Lemma 7.1]{BCLS95} with $a=p$ that 
\[
\sum_{k\in\mathbb{Z}}\mathbf{1}_{Z_k}(x,y)\rho^{kp}(\rho-1)^p\leq p|f(x)-f(y)|^p
\]
and the same holds for $Z_k^*$, hence 
\[
\sum\limits_{k\in\mathbb{Z}}J_1(k)+\sum\limits_{k\in\mathbb{Z}}J_2(k)\leq (2p+1)\int_X\int_X|f(x)-f(y)|^pp_t(y,dx)d\mu(y).
\]
Adding to these the term from~\eqref{E:L:local_norm_approx:a_h2} finally yields~\eqref{E:L:local_norm_approx:a_h1}. 
\end{proof}
\begin{remark}\label{R:local_norm_approx:a}
The previous Lemma~\ref{L:local_norm_approx:a}  corresponds to the condition $(H_p)$, $p\geq 1$, introduced in~\cite[Section 2]{BCLS95}. This fact will become specially relevant later to obtain Trudinger-Moser inequalities.
\end{remark}

\subsection{ $L^p$ Pseudo-Poincar\'e inequalities}\label{subS:PPIs}

Pseudo-Poincar\'e inequalities are a widely applicable tool to obtain Sobolev inequalities, see e.g.~\cite[Section 3.3]{SC02}. In this paragraph we introduce and discuss a pair of assumptions that will become crucial for our further analysis of Gagliardo-Nirenberg and Trudinger-Moser inequalities. Besides the corresponding $L^p$ pseudo-Poincar\'e inequalities, that are related to a weak notion of curvature (in the Bakry-\'Emery sense) of the underlying space, we will also impose certain regularity conditions on the semigroup $\{P_t\}_{t\geq 0}$.

\subsubsection{Global versions}

As with the definition of the BV and the Sobolev classes, the conditions we discuss are expressed differently in each case, which we therefore present separately.

\subsubsection*{The case $p>1$}
The two assumptions that we consider concern the validity of a $L^p$ pseudo-Poincar\'e inequality, and the continuity of the heat semigroup in a suitable Sobolev space.
\begin{itemize}[leftmargin=2em]
\item\textbf{Condition} $(\mathrm {PPI}_p)$, $p \ge 1$. There exists a constant $C_p>0$ such that for every $t \ge 0$ and $f \in W^{1,p}(\mathcal{E})$ (or $BV(\mathcal{E})$ for $p=1$),
\[
\| P_t f -f \|_{L^p(X,\mu)} \le C_p t^{\alpha_p} \mathbf{Var}_{p,\mathcal{E}} (f).
\]
\item\textbf{Condition} $(\mathrm G_q)$, $q>1$. There exists a constant $C_q>0$ such that for every $t >0$ and $f \in L^q (X,\mu)$, 
\begin{align}\label{weakBELp}
\| P_t f \|_{q,\alpha_q} \le  \frac{C_q}{t^{1-\alpha_p}} \| f \|_{L^q(X,\mu)},
\end{align} 
where $p$ is the H\"older conjugate exponent of $p$,  i.e. $\frac{1}{p}+\frac{1}{q}=1$.
\end{itemize}

\begin{remark}
It follows from spectral theory that $\alpha_2=1/2$ and that the assumptions $(\mathrm G_2)$ and $(\mathrm {PPI}_2)$ always hold, see also Theorem~\ref{T:BV3-p=2}.
\end{remark}

\begin{proposition}\label{P:Lp_pseudoPI}
Let $p>1$ and let $q$ be its H\"older conjugate. Under condition $(\mathrm G_q)$, for every $f \in W^{1,p}(\mathcal E)$ and $ t \ge 0$
\[
\| P_t f - f\|_{L^p(X,\mu)} \le \frac{C_q}{2\alpha_p} t^{\alpha_p} \mathbf{Var}_{p,\mathcal{E}} (f),
\]
where $C_q$ is the same as in~\eqref{weakBELp}. In particular, condition $(\mathrm {PPI}_p)$ is satisfied.
\end{proposition}

\begin{proof}
For any $u,v\in\mathcal{F}$ we denote
\[
\mathcal{E}_\tau(u,v)= \frac1\tau\int_X u(I-P_\tau)v\,d\mu 
 = \frac1{2\tau} \int_X\int_X p_\tau(x,y)\bigl( u(x)-u(y)\bigr)\bigl(v(x)-v(y) \bigr) d\mu(x)d\mu(y).
 \]
Fix $f\in L^p(X,\mu)$ and $h\in L^q(X,\mu)$, recalling that $p$ and $q$ are conjugate exponents. Using the strong continuity of $P_t$ in $L^1(X,\mu)$, for $t>0$ one has (see e.g. the proof of~\cite[Proposition 3.10]{ABCRST3}) 
\begin{align}
\int_X (f-P_t f  ) h\, d\mu =  \lim_{\tau \to 0^+} \int_0^t \mathcal{E}_\tau (P_s f ,h) ds. \label{eq:Pt-IbyVar1}
\end{align}
Applying H\"older's inequality and $(\mathrm G_q)$ yields
\begin{align*}
& 2\big|\mathcal{E}_\tau(f,P_sh)\bigr| \\
\leq & \frac{1}{\tau}\int_X \int_X p_\tau (x,y) | P_s h(x)-P_sh(y)| \, |f(x)-f(y)| d\mu(x) d\mu(y)\,ds \notag \\
\leq &  \frac{1}{\tau} \biggl(\int_X\int_X p_\tau(x,y)| P_s h(x)-P_sh(y)|^q\,d\mu(x)\,d\mu(y)\biggr)^{1/q} \biggl(\int_X\int_X p_\tau(x,y)|f(x)-f(y)|^p\,d\mu(x)\,d\mu(y)\biggr)^{1/p} \notag \\ 
\leq & C_q  \tau^{-\alpha_p} s^{-(1-\alpha_p)} \|h\|_{L^q(X,\mu)}  \biggl(\int_X\int_X p_\tau(x,y)|f(x)-f(y)|^p\,d\mu(x)\,d\mu(y)\biggr)^{1/p}.
\end{align*}
Integrating over $s\in(0,t)$ and taking $\liminf_{\tau\to0^+}$ gives for $p >1 $
\begin{align*}
\biggl| \int_X (f-P_t f  ) h d\mu \biggr| \le \frac{C_q}{2} \frac{t^{\alpha_p}}{\alpha_p}  \| h \|_{L^q(X,\mu)} \mathbf{Var}_{p,\mathcal{E}} (f)
\end{align*}
and the conclusion follows by $L^p$-$L^q$ duality.
\end{proof}

\subsubsection*{The case $p=1$}

Recall that the semigroup $\{P_t\}_{t\geq 0}$ admits a measurable heat kernel $p_t(x,y)$ because $(X,\mu)$ is assumed to be a good measurable space, c.f.~\cite[Theorem 1.2.3]{BGL14}). In addition, we consider the space $(X,\mu)$ to be endowed with a metric $d$. This metric $d$ does not need to be intrinsically associated with the Dirichlet form but has to satisfy some conditions listed below. The pseudo-Poincar\'e inequality considered for $p>1$ is now replaced by an 
inequality for the heat kernel, whereas the regularity condition on the heat semigroup in this case is spelled in terms of its H\"older continuity.
\begin{itemize}[leftmargin=2em]
\item\textbf{Condition}. 
For any $\kappa \ge 0$, there exist constants $C,c>0$ such that for every $t>0$ and a.e. $x,y \in X$
\begin{align}\label{E:weak_Harnack}
d(x,y)^\kappa p_t(x,y) \le C t^{\kappa/d_W} p_{ct}(x,y),
\end{align}
where $d_W>1$ is a parameter independent from $\kappa, C$ and $c$.
\item\textbf{Condition} $(\mathrm G_\infty)$. There exists a constant $C>0$ so that for every $t >0$, $f \in L^\infty(X,\mu)$, and $x,y \in X$
\begin{align}\label{weakBE1}
| P_t f (x)-P_t f(y)| \le C \frac{d(x,y)^{d_W(1-\alpha_1)}}{t^{1-\alpha_1}} \| f \|_{L^\infty(X,\mu)}.
\end{align} 
\end{itemize}
We note that~\eqref{E:weak_Harnack} is for instance satisfied if $p_t(x,y)$ satisfies sub-Gaussian heat kernel estimates, see~\cite[Lemma 2.3]{ABCRST3} 
and that 
the condition $(\mathrm G_\infty)$ was called in~\cite{ABCRST3} the weak Bakry-\'Emery estimate.
 
\begin{remark}\label{R:wBE_to_Gq}
Since $(\mathrm G_2)$ always holds for every $t>0$, using interpolation theory, one deduces as in the proof of~\cite[Theorem 3.9]{ABCRST3} that the assumption $(\mathrm G_\infty)$ implies that for every $t >0$,  $q \ge 2$ and $f \in L^p(X,\mu)$, 
\begin{equation}\label{E:Lp_bdd}
\| P_t f \|_{q,\beta_q} \le  \frac{C_q}{t^{\beta_q}} \| f \|_{L^q(X,\mu)},
\end{equation}
where $\beta_q=\big( 1-\frac{2}{q}\big)(1-\alpha_1)+\frac{1}{q}$. This is not quite the same as $(\mathrm G_q)$, unless $1-\alpha_p=\beta_q$, i.e $\alpha_p=\big( 1-\frac{2}{p}\big)(1-\alpha_1)+\frac{1}{p}$. Note that for the Vicsek set (or direct products of it) one indeed has $\alpha_p=\big( 1-\frac{2}{p}\big)(1-\alpha_1)+\frac{1}{p}$, see Remark~\ref{R:Vicsek2}.
\end{remark}

\begin{proposition}\label{PseudoP}
If the Dirichlet space $(X,d,\mu,\mathcal E)$ satisfies $(\mathrm G_\infty)$and~\eqref{E:weak_Harnack}, there exists a constant $C>0$ such that for every $f \in BV(\mathcal E)$ and $ t \ge 0$,
\[
\| P_t f - f\|_{L^1(X,\mu)} \le C t^{\alpha_1} \mathbf{Var}_\mathcal{E} (f).
\]
In particular $(\mathrm {PPI}_1)$ is satisfied.
\end{proposition}

\begin{proof}
See~\cite[Proposition 3.10]{ABCRST3}.
\end{proof}

\subsubsection{Localized versions}

We finish this section with the local counterparts of the previous conditions since these shall ultimately be used to obtain the whole family of inequalities in the subsequent sections. 
 
\begin{itemize}[leftmargin=2em]
\item\textbf{Condition} $(\mathrm {PPI}_p(R))$, $p \ge 1$. There exists a constant $C_p(R)>0$ such that for every $t \in (0,R)$ and $f \in W^{1,p}(\mathcal{E})$ (or $BV(\mathcal{E})$ for $p=1$),
\[
\| P_t f -f \|_{L^p(X,\mu)} \le C_p(R) t^{\alpha_p} \mathbf{Var}_{p,\mathcal{E}} (f).
\]
\item\textbf{Condition} $(\mathrm G_q(R))$, $q>1$, $R>0$. There exists a constant $C_q(R)>0$ such that for every $t \in (0,R)$ and $f \in L^q (X,\mu)$, 
\begin{equation}\label{weakBELpR}
\| P_t f \|_{q,\alpha_q} \le  \frac{C_q}{t^{1-\alpha_p}} \| f \|_{L^q(X,\mu)},
\end{equation}
where as before $p$ is the H\"older conjugate exponent of $p$,  i.e. $\frac{1}{p}+\frac{1}{q}=1$.
\end{itemize}

The same proof as Proposition~\ref{P:Lp_pseudoPI} yields the following result.

\begin{proposition}
Let $p>1$, $R>0$ and assume that $(\mathrm G_q(R))$ holds, where $q$ is the H\"older conjugate of $p$. Then, for every $f \in W^{1,p}(\mathcal E)$ and $t \in (0,R)$,
\begin{equation*}
\| P_t f - f\|_{L^p(X,\mu)} \le \frac{C_q(R)}{2\alpha_p} t^{\alpha_p} \mathbf{Var}_{p,\mathcal{E}} (f)
\end{equation*}
with the same constant $C_q$ as in~\eqref{weakBELpR}. In particular, $(\mathrm {PPI}_p(R))$ is satisfied.
\end{proposition}

Similarly, to treat the case $p=1$ one can introduce a localized version of~\eqref{E:weak_Harnack} and of the condition $(\mathrm G_\infty(R))$, $R>0$ to prove the localized analogue of Proposition~\ref{PseudoP}. We omit the details for conciseness.

\subsection{Weak Bakry-\'Emery estimates}

As shown in the previous section, uniform regularization properties of the semigroup $\{P_t\}_{t\geq 0}$ play an important role in our study because they yield pseudo-Poincar\'e type estimates for the semigroup. In this section, we investigate some self-improvement properties of the assumption $(\mathrm G_\infty(R))$, $R>0$.

\begin{lemma}\label{L:local_wBECD1}
Let $d$ be a metric on $X$. Let $R>0$ and assume that there exist constants $C,\kappa,d_W>0$ such that for every $t \in (0,R)$, $f\in L^\infty(X,\mu)$ and $x,y \in X$,
\begin{equation}\label{E:local_wBECD1}
| P_t f (x)-P_t f(y)| \le C \frac{d(x,y)^\kappa}{t^{\kappa/d_W}} \| f \|_{ L^\infty(X,\mu)}.
\end{equation}
Then, for any $R' \ge R$,~\eqref{E:local_wBECD1} also holds for every $t \in (0,R')$ with a possibly different constant $C=C_{R'}$.
\end{lemma}

\begin{proof}
Let $f\in L^\infty(X,\mu)$ and $x,y \in X$. We use an argument from~\cite{BK19} and prove first by induction that, for any $t\in (0,R)$ and $n\in\mathbb{N}$
\begin{equation}\label{E:local_wBECD1_01}
| P_{nt} f (x)-P_{nt} f(y)| \le C_R2^{\frac{(n-1)\kappa}{d_W}}\frac{d(x,y)^\kappa}{(nt)^{\kappa/d_W} }\| f \|_{L^\infty (X,\mu)}.
\end{equation}
For $n=1$ this is assumption~\eqref{E:local_wBECD1}. Now, due to the semigroup property and the contractivity of $\{P_t\}_{t>0}$ we get
\begin{align*}
|P_{(n+1)t}f(x)-P_{(n+1)t}&f(y)|=|P_{nt}(P_t f)(x)-P_{nt}(P_t f)(y)|\leq C_R2^{\frac{(n-1)\kappa}{d_W}} \frac{d(x,y)^\kappa}{(nt)^{\kappa/d_W}}\|P_t f\|_{L^\infty(X,\mu)}\\
&\leq C\frac{d(x,y)^\kappa}{(nt)^{\kappa/d_W}}\| f\|_{L^\infty(X,\mu)}=\frac{C_R2^{\frac{(n-1)\kappa}{d_W}} d(x,y)^\kappa}{(nt+t)^{\kappa/d_W}}\bigg(\frac{nt+t}{nt}\bigg)^{\kappa/d_W}\!\!\|f\|_{L^\infty(X,\mu)}\\
&\le C_R2^{n\kappa/d_W}\frac{d(x,y)^\kappa}{(nt+t)^{\kappa/d_W}}\|f\|_{L^\infty(X,\mu)}.
\end{align*}
Finally, for any $R\leq t< R'$ there is $n\in\mathbb{N}$ and $s\in(0,R)$ such that $t=ns$, hence~\eqref{E:local_wBECD1_01} yields~\eqref{E:local_wBECD1} with a suitable constant.
\end{proof}

To extend~\eqref{E:local_wBECD1} to all of $t>0$ requires a better (uniform) control on the constants, which is possible under additional conditions.

\begin{lemma}\label{L:local_wBECD2}
Let $d$ be a metric on $X$. Let $R>0$ and assume that there exist constants $C,\kappa,d_W>0$ such that for every $t \in (0,R)$, $f\in L^\infty(X,\mu)$ and $x,y \in X$,
\begin{equation}\label{E:local_wBECD2}
| P_t f (x)-P_t f(y)| \le C \frac{d(x,y)^\kappa}{t^{\kappa/d_W}} \| f \|_{ L^\infty(X,\mu)}.
\end{equation}
Moreover, assume that
\begin{enumerate}[wide=0em,leftmargin=1.5em,label={\rm (\roman*)}]
\item the infinitesimal generator $\Delta$ of the Dirichlet form $(\mathcal E,\mathcal{F})$ has a pure point spectrum, 
\item $1 \in \mathbf{dom}\, \Delta $,
\item the Dirichlet space $(X,\mu,\mathcal E, \mathcal{F})$ satisfies the Poincar\'e inequality
\[
\int_X\Big( f-\int_X f d\mu\Big)^2\!d\mu\,\le \frac{1}{\lambda_1} \mathcal E (f,f)
\]
for some $\lambda_1>0$ and all  $f \in \mathcal{F}$,
\item the heat kernel $p_t(x,y)$ of $P_t$ satisfies the estimate
\[
p_{t_0}(x,y) \le M
\]
for some $ t_0 ,M>0$ and $\mu$-almost every $x,y \in X$.
\end{enumerate}
Then,~\eqref{E:local_wBECD2} holds for all $t>0$, possibly with a different constant $C>0$.
\end{lemma}

\begin{proof}
By virtue of assumption (ii) one has $\mu (X) <+\infty$, so that without loss of generality we can assume $\mu(X)=1$. 
Let $\{\lambda_j\}_{j\geq 0}$ denote the eigenvalues of $\Delta$ and $\{\phi_j\}_{j\geq 0}$ the associated eigenfunctions. Assumptions (ii) and (iii), see e.g.~\cite[Proposition 3.1.6]{BGL14}, 
yield for any $f \in L^2(X,\mu)$
\begin{equation}\label{E:local_wBECD2_01}
P_t f (x)=\int_X f d\mu+\sum_{j=1}^{+\infty} e^{-\lambda_j t} \phi_j (x) \int_X \phi_j(y) f (y) d\mu(y).
\end{equation}
Now, since $P_{t_0} \phi_j=e^{-\lambda_j t_0} \phi_j$, applying H\"older's inequality and assumption (iv) we deduce for $\mu$-a.e. $x\in X$
\begin{align*}
|\phi_j (x)|&=e^{\lambda_j t_0}\left|  \int_X p_{t_0} (x,y) \phi_j (y) d\mu(y) \right|  
 \le e^{\lambda_j t_0} \left( \int_X p_{t_0} (x,y)^2 d\mu (y) \right)^{1/2} 
 \le M e^{\lambda_j t_0}.
\end{align*}
Next, using if needed Lemma~\ref{L:local_wBECD1}, we may assume $t_0 \le R$. Applying~\eqref{E:local_wBECD2} to $\phi_j$ and the latter estimate we obtain
\begin{equation*}
| e^{-\lambda_j t_0} \phi_j (x) -e^{-\lambda_j t_0}  \phi_j (y)| 
 \le C M \frac{d(x,y)^\kappa}{t_0^{\kappa/d_W}}  e^{\lambda_j t_0}
\end{equation*}
hence
\begin{equation}\label{E:local_wBECD2_02}
|  \phi_j (x) - \phi_j (y)|  \le C M \frac{d(x,y)^\kappa}{t_0^{\kappa/d_W}}  e^{2\lambda_j t_0}.
\end{equation}
Finally, for any $f \in L^\infty(X,\mu)$ and $t >2 t_0$,~\eqref{E:local_wBECD2_01} and~\eqref{E:local_wBECD2_02} imply
\begin{align*}
| P_t f (x) -P_t f(y) |  &\le  \sum_{j=1}^{+\infty} e^{-\lambda_j t} | \phi_j (x) -\phi_j(y) | \int_X \phi_j(z) f (z) d\mu(z) \\
&  \le   C M \frac{d(x,y)^\kappa}{t_0^{\kappa/d_W}}  \sum_{j=1}^{+\infty} e^{-\lambda_j (t-2t_0)}  \int_X \phi_j(z) f (z) d\mu(z) \\
& \le C M \frac{d(x,y)^\kappa}{t_0^{\kappa/d_W}} \| f \|_{L^\infty(X,\mu)}   \sum_{j=1}^{+\infty} e^{-\lambda_j (t-2t_0)}   \\
& \le C' \frac{d(x,y)^\kappa}{t^{\kappa/d_W}} \| f \|_{ L^\infty(X,\mu)},
\end{align*}
where the constant $C'$ in the last inequality depends on $M,C, \kappa, d_W, \lambda_j$ and $t_0$.
\end{proof}

\section{Examples of heat semigroup  based BV and Sobolev classes}\label{sec: examples}

To illustrate the scope of our results we now present several classes of Dirichlet spaces that appear in the literature for which the heat semigroup based BV and Sobolev classes can be characterized. This generalizes previous results from~\cite{ABCRST1,ABCRST2,ABCRST3}.

\subsection{Metric measure spaces with Gaussian heat kernel estimates}\label{sec:mmsdoubling}

Further details to this particular framework can be found in~\cite{ABCRST2}. 
We consider $(X,d,\mu,\mathcal{E},\mathcal{F})$ to be a strictly local Dirichlet space, where $d$ is the intrinsic metric associated to the Dirichlet form. The measure $\mu$ is assumed to be doubling and the space to supports a scale invariant 2-Poincar\'e inequality on balls; according to K.T. Sturm's results~\cite{Stu95, Stu96} these conditions are equivalent to the fact that there is a heat kernel with Gaussian estimates. In this setting, see~\cite[Lemma 2.11]{ABCRST2}, $\mathcal E$ admits a carr\'e du champ operator $\Gamma(f,f)$, $ f\in \mathcal F$ and we denote $| \nabla f |=\sqrt{\Gamma(f,f)}$. 
Based on the ideas of M. Miranda~\cite{Mir03}, the following definitions were introduced in~\cite{ABCRST2}.

\begin{definition}[BV space]
We say that $f\in L^1(X,\mu)$ is in $BV(X)$ if there is a sequence of local Lipschitz functions $f_k\in L^1(X,\mu)$
such that $f_k\to f$ in $L^1(X,\mu)$ and 
\[
\| Df  \|(X):=\liminf_{k\to\infty}\int_X|\nabla f_k|\, d\mu<\infty.
\]
\end{definition}

\begin{definition}[Sobolev space]
For $p \ge 1$, we define the Sobolev space
\begin{equation}\label{definition:W1p}
W^{1,p}(X):=\left\{ f \in L^p(X,\mu)\cap\mathcal{F}_{\mathrm{loc}}(X)\, :\,   |\nabla f|\in L^p(X) \right\}
\end{equation}
whose norm is given by $\|f\|_{W^{1,p}(X)}=\|f\|_{L^p(X,d\mu)}+\|\,|\nabla f|\,\|_{L^p(X,\mu)}$.
\end{definition}

\begin{theorem}\label{T:BV2}
For each $R \in (0,+\infty]$ the following holds:
\begin{enumerate}[wide=0em,label={\rm (\roman*)}]
\item Assume the weak Bakry-\'Emery estimate
\begin{equation}\label{E:wBE_BV2}
\|\,|\nabla P_tf|\,\|_{L^\infty(X,\mu)}\leq \frac{C}{\sqrt{t}}\|f\|_{L^\infty(X,\mu)}\qquad t \in (0,R)
\end{equation}
for some constant $C>0$ and any $f\in\mathcal{F}\cap L^\infty(X,\mu)$. Then, $(\mathrm {PPI}_1(R))$ is satisfied, $\alpha_1=\frac{1}{2}$, $BV(\mathcal{E})=BV(X)$ and
\[
 \mathbf{Var}_{\mathcal{E}} (f)  \simeq\| f \|_{1,1/2,R}\simeq  \liminf_{r\to 0^+} \int_X\int_{B(x,r)}\frac{|f(y)-f(x)|}{ \sqrt{r}\mu(B(x,r))}\, d\mu(y)\, d\mu(x) \simeq \| Df  \|(X).
\]
\item Assume the \textit{quasi Bakry-\'Emery} condition
estimate, c.f.~\cite[Definition 2.15]{ABCRST2},
\begin{equation}\label{E:sBE_BV2}
|\nabla P_tf|\leq  C P_t|\nabla f|\qquad t \in (0,R)
\end{equation}
$\mu$-a.e. for some constant $C>0$ and any $f\in\mathcal{F}$. Then, for every $p>1$, condition $(\mathrm {PPI}_p(R))$ is satisfied, $\alpha_p=\frac{1}{2}$, $W^{1,p}(\mathcal{E})=W^{1,p}(X)$ and
\[
 \mathbf{Var}_{p,\mathcal{E}} (f)  \simeq \| f \|_{p,1/2,R} \simeq\left(\int_X | \nabla f |^p d\mu\right)^{1/p}.
\]
\end{enumerate}
\end{theorem}

\begin{proof}
It suffices to show the statements for non-negative functions.
\begin{enumerate}[wide=0em,label={\rm (\roman*)}]
\item Let $f\in BV(X)$ non-negative. With the same proof as in~\cite[Lemma 4.3]{ABCRST2}, condition~\eqref{E:wBE_BV2} implies that
\begin{equation*}
\|P_tf-f\|_{L^1(X,\mu)}\leq C\sqrt{t}\int_X|\nabla f|\,d\mu
\end{equation*}
for any $t\in (0,R)$. Analogous to the proof of~\cite[Theorem 4.4]{ABCRST2}, the latter inequality and the coarea formula~\cite[Theorem 3.11]{ABCRST2} yield
\[
\frac{1}{\sqrt{t}}\int_X\int_X|f(x)-f(y)|\,p_t(x,y)\,d\mu(y)\,d\mu(x)\leq 2C\|Df\|(X)
\]
for any $t\in (0,R)$ and hence $\mathbf{Var}_{\mathcal{E}} (f)\leq \| f \|_{1,1/2,R}\leq 2C\|Df\|(X)$. Let us now assume $f\in BV(\mathcal{E})$. Due to the Gaussian lower bound of the heat kernel, for any $t\in (0,R)$ we have
\begin{align*}
&\frac{1}{\sqrt{t}}\int_X\int_X|f(x)-f(y)|\,p_t(x,y)\,d\mu(y)\,d\mu(x)\\
&\geq \frac{1}{\sqrt{t}}\int_X\int_X|f(x)-f(y)|\frac{e^{-c\frac{d(x,y)^2}{t}}}{C\mu(B(x,\sqrt{t})}\,d\mu(y)\,d\mu(x)\\
&\geq \frac{1}{\sqrt{t}}\int_X\int_{B(x,\sqrt{t})}|f(x)-f(y)|\frac{e^{-c\frac{d(x,y)^2}{t}}}{C\mu(B(x,\sqrt{t})}\,d\mu(y)\,d\mu(x)\\
&\geq \frac{C}{\sqrt{t}}\int_X\int_{B(x,\sqrt{t})}\frac{|f(x)-f(y)|}{\mu(B(x,\sqrt{t})}\,d\mu(y)\,d\mu(x).
\end{align*}
Taking $\liminf_{t\to 0^+}$ on both sides of the inequality we get
\[
\mathbf{Var}_{\mathcal{E}} (f)\geq \liminf_{t\to 0^+}\frac{C}{\sqrt{t}}\int_X\int_{B(x,\sqrt{t})}\frac{|f(x)-f(y)|}{\mu(B(x,\sqrt{t})}\,d\mu(y)\,d\mu(x)\geq C\|Df\|(X),
\]
where the last inequality follows from the second part of the proof of~\cite[Theorem 3.1]{MMS16} (which does not use 1-Poincar\'e inequality). One now readily gets $\alpha_1=1/2$.
\item
Let $f\in \mathbf{B}^{p,1/2}(X)$. As in the proof of~\cite[Theorem 4.11]{ABCRST2}, for any $t\in (0,R)$ it holds that
\begin{equation}\label{E:BV2_02}
\int_X|\nabla f_{t/2}(x)|^pd\mu(x)\leq \frac{C}{t^p}\int_X\int_{B(x,t)}\frac{|f(x)-f(y)|^p}{\mu(B(x,t))}\,d\mu(y)\,d\mu(x)\leq C\|f\|_{p,1/2,R}^p,
\end{equation}
where $f_t:=\sum_{i\geq 1} f|_{B_i^t}\varphi_i^t$, $\{B_i^t\}_{i\geq 1}$ is a suitable covering of $X$ and $\{\varphi_i^t\}_{i\geq 1}$ a subordinated $(C/t)$-Lipschitz partition of unity. Following further~\cite[Theorem 4.11]{ABCRST2}, we also get
\begin{multline*}
\int_X|f_{t/2}(x)-f(x)|^pd\mu(x)
\leq C t^p\int_X\int_{B(x,t)}\frac{|f(x)-f(y)|^p}{t^p\mu(B(x,t))}\,d\mu(y)\,d\mu(x)\\
\leq C t^p\|f\|_{p,1/2,R}^p
\end{multline*}
which in particular implies $\|f_t-f\|_{L^p(X,\mu)}\to 0$ as $t\to 0$. Let us now consider $\{t_n\}_{n\geq 0}$ with $t_n\to 0$. In view of~\eqref{E:BV2_02}, the sequence $\{|\nabla f_{t_n}|\}_{n\geq 0}$ is uniformly bounded in $L^p$ and since the latter is a reflexive space, we find a subsequence that is weakly convergent in $L^p$. By virtue of Mazur's theorem, see e.g.~\cite[p.120]{Yos95}, one can extract a sequence of convex combinations of $\{|\nabla f_{t_n}|\}_{n\geq 0}$ that converges in $L^p$. The corresponding convex combinations of $\{|f_{t_n}|\}_{n\geq 0}$ thus converge to $f$ on ${W^{1,p}(X)}$ as $n\to\infty$ and hence $\||\nabla f_{t_n}|-|\nabla f|\|_{L^p(X,\mu)}\to 0$. Finally, taking $\liminf_{t\to 0^+}$ in both sides of the first inequality of~\eqref{E:BV2_02} yields 
\[
\|\,|\nabla f|\,\|_{L^p(X,\mu)}\leq C\mathbf{Var}_{p,\mathcal{E}}(f) \leq C \|f\|_{p,1/2,R}.
\] 
To obtain the reverse inequality, let us assume that $f\in L^p(X,\mu)\cap\mathcal{F}$ with $|\nabla f|\in L^p(X,\mu)$. Following verbatim the proof of~\cite[Theorem 4.17]{ABCRST2} with $t\in (0,R)$, the quasi Bakry-\'Emery condition~\eqref{E:sBE_BV2} implies
\[
\frac{1}{\sqrt{t}}\bigg(\int_XP_t(|f-f(x)|^p)(x)\,\mu(dx)\bigg)^{1/p}\leq 2C\|\,|\nabla f|\,\|_{L^p(X,\mu)},
\]
with a constant $C>0$ independent of $R$. Taking $\liminf_{t\to 0^+}$ in both sides of the inequality we obtain $\mathbf{Var}_{p,\mathcal{E}}(f) \leq 2C\|\,|\nabla f|\,\|_{L^p(X,\mu)}$. The result extends to any $f\in W^{1,p}(X)$ exactly as in the proof of~\cite[Theorem 4.17]{ABCRST2} 
and in particular $\alpha_p=1/2$.
\end{enumerate}
\end{proof}

\subsubsection*{About the Bakry-\'Emery conditions}

As one would expect, the quasi Bakry-\'Emery 
curvature condition~\eqref{E:sBE_BV2} implies the weak one~\eqref{E:wBE_BV2}. Examples of spaces within the framework just discussed that satisfy~\eqref{E:sBE_BV2} include Riemannian manifolds with  Ricci curvature bounded from below and   $RCD(K,+\infty)$ spaces; in that case for every $t \ge 0$, $|\nabla P_tf|\leq  e^{-Kt} P_t|\nabla f|$, and thus $|\nabla P_tf|\leq  C P_t|\nabla f|$ for  $t \in (0,R)$ with $C=\max (1, e^{-KR}) $, see~\cite{MR3121635}. On the other hand, Carnot groups~\cite{BB16} and complete sub-Riemannian manifolds with generalized Ricci curvature bounded from below in the sense of~\cite{BG17,BBG14} are examples in this setting where the weak Bakry-\'Emery condition~\eqref{E:wBE_BV2} is known but the stronger condition~\eqref{E:sBE_BV2} unknown.

\subsection{Metric measure spaces  with sub-Gaussian heat kernel estimates}

In this subsection, we consider $(X,d,\mu,\mathcal{E},\mathcal{F})$ to be a strongly local metric Dirichlet space for which 
balls of finite radius have compact closure. In contrast to~\cite{ABCRST3}, the metric measure space $(X,d,\mu)$ need \textit{not} be Ahlfors regular. 
The semigroup $\{P_t\}_{t>0}$ is assumed to have a continuous heat kernel $p_t(x,y)$ satisfying estimates
\begin{multline}\label{E:subGaussian_HKE}
\frac{c_{1}}{\mu(B(x,t^{1/d_W}))} \exp\biggl(-c_{2}\Bigl(\frac{d(x,y)^{d_{W}}}{t}\Bigr)^{\frac{1}{d_{W}-1}}\biggr) \\
\le p_{t}(x,y)\leq  \frac{c_{3}}{\mu(B(x,t^{1/d_W}))}\exp\biggl(-c_{4}\Bigl(\frac{d(x,y)^{d_{W}}}{t}\Bigr)^{\frac{1}{d_{W}-1}}\biggr)
\end{multline}
for $\mu$-a.e. $x,y\in X$ and each $t>0$, where $c_1,c_2, c_3, c_4 >0$ and $d_{W}\geq 2$. The exact values of $c_1,c_2,c_3,c_4$ are irrelevant in our analysis, however the parameter $d_W$, called the walk dimension of the space, will play an important role. In general, when $d_W=2$ one speaks of Gaussian estimates and when $d_W > 2$ of sub-Gaussian estimates. Notice that~\eqref{E:subGaussian_HKE} is also valid when $X$ is compact, as for instance the standard Sierpinski gasket or Vicsek set; in that case, for large times $t$ the ball $B(x,t^{1/d_W})$ fills the space and only the exponential term remains. We also note that~\eqref{E:subGaussian_HKE} implies the estimate~\eqref{E:weak_Harnack}, see~\cite[Lemma 2.3]{ABCRST3}, and that the measure $\mu$ is doubling.

\subsubsection*{The case $p>1$}

The following metric characterization of the Sobolev spaces is available for $p>1$.

\begin{theorem}\label{W1p:I}
For any $p>1$,
\[
W^{1,p}(\mathcal E)=\Big\{ f \in L^p(X,\mu), \,  \limsup_{r\to 0^+} \frac{1}{r^{\alpha_pd_W}}\Big( \int_X\int_{B(x,r)}\frac{|f(y)-f(x)|^p}{ \mu(B(x,r))}\, d\mu(y)\, d\mu(x)\Big)^{1/p} \Big\}.
\]
Moreover, the $p$-variation of any $f\in W^{1,p}(\mathcal E)$ can be bounded by
\begin{equation*}
\mathbf{Var}_{p, \mathcal{E}} (f)   
 \ge c  \liminf_{r\to 0^+} \frac{1}{r^{\alpha_pd_W}}\Big( \int_X\int_{B(x,r)}\frac{|f(y)-f(x)|^p}{ \mu(B(x,r))}\, d\mu(y)\, d\mu(x)\Big)^{1/p} 
\end{equation*}
and
\begin{equation*}
\mathbf{Var}_{p, \mathcal{E}} (f)   
 \le C  \limsup_{r\to 0^+} \frac{1}{r^{\alpha_pd_W}}\Big( \int_X\int_{B(x,r)}\frac{|f(y)-f(x)|^p}{ \mu(B(x,r))}\, d\mu(y)\, d\mu(x)\Big)^{1/p} .
\end{equation*}
\end{theorem}

\begin{proof}
By virtue of the sub-Gaussian lower estimate~\eqref{E:subGaussian_HKE}, for any $s,t>0$ and $\alpha>0$ we have
\begin{align*}
 \int_X \int_X & |f(x)-f(y) |^p p_t (x,y) d\mu(x) d\mu(y) \\
 \ge & \int_X \int_{B(y,s)} |f(x)-f(y) |^p p_t (x,y) d\mu(x) d\mu(y) \\
 \ge &c_{1} \int_X \int_{B(y,s)} \frac{ |f(x)-f(y) |^p}{\mu(B(y,t^{1/d_W}))}  \exp\biggl(-c_{2}\Bigl(\frac{d(x,y)^{d_{W}}}{t}\Bigr)^{\frac{1}{d_{W}-1}}\biggr) d\mu(x) d\mu(y) \\
 \ge & c_{1} \exp\biggl(-c_{2}\Bigl(\frac{s^{d_{W}}}{t}\Bigr)^{\frac{1}{d_{W}-1}}\biggr) \int_X \int_{B(y,s)}\frac{ |f(x)-f(y) |^p}{\mu(B(y,t^{1/d_W}))}  d\mu(x) d\mu(y).
\end{align*}
Choosing $t=s^{d_W}$ and dividing on both sides of the inequality by $s^{p\alpha_pd_W}$ lead to
\begin{equation}\label{E:W1p:I_01}
\frac{1}{s^{p\alpha_p d_W}} \int_X \int_{B(y,s)} \frac{ |f(x)-f(y) |^p}{\mu(B(y,s))}  d\mu(x) d\mu(y) \le C \frac{1}{s^{p \alpha_p d_W}} \int_X \int_X  |f(x)-f(y) |^p p_{s^{d_W}} (x,y) d\mu(x) d\mu(y)
\end{equation}
which in view of the definition of $W^{1,p}(\mathcal E)$ implies
\[
W^{1,p}(\mathcal E)\subset \Big\{ f \in L^p(X,\mu), \,  \limsup_{r\to 0^+} \frac{1}{r^{\alpha_pd_W}}\Big( \int_X\int_{B(x,r)}\frac{|f(y)-f(x)|^p}{ \mu(B(x,r))}\, d\mu(y)\, d\mu(x)\Big)^{1/p}<\infty \Big\}.
\]
Moreover, taking $\liminf_{s \to 0^+}$ on both sides of~\eqref{E:W1p:I_01} yields the lower bound
\[
\mathbf{Var}_{p, \mathcal{E}} (f)   \ge c
 \liminf_{r\to 0^+} \frac{1}{r^{\alpha_pd_W}}\Big( \int_X\int_{B(x,r)}\frac{|f(y)-f(x)|^p}{ \mu(B(x,r))}\, d\mu(y)\, d\mu(x)\Big)^{1/p}.
\]
The converse estimate is more difficult to prove and we shall argue somewhat similarly to the proof of~\cite[Lemma 4.13]{ABCRST3}.
Let us denote
\begin{equation*}
\Psi (t):=\frac{1}{ t^{p\alpha_p}}\int_X\int_X p_{t}(x,y) |f(x)-f(y)|^p \,d\mu(x)\,d\mu(y)
\end{equation*}
and proceed as follows: Fix $\delta>0$ and set $r=\delta t^{1/d_W}$. For $d(x,y) \le \delta t^{1/d_W}$ the sub-Gaussian upper bound~\eqref{E:subGaussian_HKE} implies $p_t(x,y)\leq \frac{C}{\mu(B(x,t^{1/d_W}))}$, so that
\begin{multline*}
\frac{1}{ t^{p\alpha_p}}\int_X\int_{B(y,r)} p_{t}(x,y) |f(x)-f(y)|^p \,d\mu(x)\,d\mu(y) \\	
\leq \frac{C}{t^{p\alpha_p}} \int_X\int_{B(y,\delta t^{1/d_W})} \frac{ |f(x)-f(y) |^p}{\mu(B(x,\delta t^{1/d_W}))}  \,d\mu(x)\,d\mu(y)
:=\Phi(t).
\end{multline*}
For $d(x,y)>\delta t^{1/d_W}$, we instead use the sub-Gaussian upper bound to see there are $c,C>1$ (independent of $\delta$) such that
\begin{equation*}
p_t(x,y)\leq C\exp\biggl( -\Bigl(\frac{c_4}2\Bigr)\Bigl(\frac{d(x,y)^{d_W}}t\Bigr)^{\frac1{d_W-1}}\biggr) p_{ct}(x,y)
\leq C\exp\bigl(-c' \delta^{\frac{d_W}{d_W-1}}\bigr) p_{ct}(x,y).
\end{equation*}
Therefore,
\begin{align}\label{ref:op}
	\Psi(t)
	&\leq \Phi(t) + \frac{1}{ t^{p\alpha_p}}\int_X\int_{X\setminus B(y,r)} p_{t}(x,y) |f(x)-f(y)|^p \,d\mu(x)\,d\mu(y)\notag\\
	&\leq \Phi(t)+ \frac{C}{t^{p\alpha_p}}\exp\bigl(-c' \delta^{\frac{d_W}{d_W-1}}\bigr) \int_X\int_{X\setminus B(y,r)} p_{ct}(x,y) |f(x)-f(y)|^p \,d\mu(x)\,d\mu(y)\notag\\
	&\leq \Phi(t)+ \frac{C}{t^{p\alpha_p}}\exp\bigl(-c' \delta^{\frac{d_W}{d_W-1}}\bigr) \int_X\int_{X} p_{ct}(x,y) |f(x)-f(y)|^p \,d\mu(x)\,d\mu(y)\notag \\
	&= \Phi(t) + A_\delta \Psi(ct),
\end{align}
where $A_\delta$ is a constant that can be made as small as we desire by choosing $\delta$ large enough. Letting first $t \to 0^+$ one gets
\begin{multline*}
 \limsup_{t \to 0^+} \frac{1}{ t^{p\alpha_p}}\int_X\int_{X} p_{t}(x,y) |f(x)-f(y)|^p \,d\mu(x)\,d\mu(y) \\
 \le C \delta^{p \alpha_p d_W} \limsup_{r\to 0^+} \frac{1}{r^{p \alpha_pd_W}} \int_X\int_{B(x,r)}\frac{|f(y)-f(x)|^p}{ \mu(B(x,r))}\, d\mu(y)\, d\mu(x) \\
  + CA_\delta \limsup_{t \to 0^+} \frac{1}{ t^{p\alpha_p}}\int_X\int_{X} p_{t}(x,y) |f(x)-f(y)|^p \,d\mu(x)\,d\mu(y) 
\end{multline*}
and choosing $\delta$ large enough the conclusion follows.
\end{proof}

The case $p=2$ is special and allows to improve and extend the previous result.

\begin{theorem}\label{T:BV3-p=2}
The property $(\mathrm {PPI}_2)$ is always satisfied, $\alpha_2=1/2$, and the following equivalences of semi-norms is valid on $ W^{1,2}(\mathcal E)$:
\begin{align*}
\mathbf{Var}_{2,\mathcal{E}} (f)   &\simeq \mathcal{E}(f,f) \\
& \simeq \liminf_{r\to 0^+} \frac{1}{r^{d_W/2}}\left( \int_X\int_{B(x,r)}\frac{|f(y)-f(x)|^2}{ \mu(B(x,r))}\, d\mu(y)\, d\mu(x)\right)^{1/2}\\
 & \simeq \sup_{r >0}\frac{1}{r^{d_W/2}}\left( \int_X\int_{B(x,r)}\frac{|f(y)-f(x)|^2}{ \mu(B(x,r))}\, d\mu(y)\, d\mu(x)\right)^{1/2} \\
 &\simeq \| f \|_{2,1/2}
\end{align*}
\end{theorem}

\begin{proof}
The fact that $\alpha_2=1/2$ is proved in~\cite[Proposition 5.6]{ABCRST1}, which together with~\cite[Lemma 4.20]{ABCRST1} yields property $(\mathrm {PPI}_2)$.
Now, the equivalence of semi-norms
\begin{equation*}
\mathbf{Var}_{2,\mathcal{E}} (f) \simeq \mathcal{E}(f,f) \simeq \| f \|_{2,1/2}
\end{equation*}
follows from~\cite[Proposition 4.6]{ABCRST1} and
\begin{equation*}
\mathbf{Var}_{2,\mathcal{E}} (f) \simeq \sup_{r >0}\frac{1}{r^{d_W/2}}\left( \int_X\int_{B(x,r)}\frac{|f(y)-f(x)|^2}{ \mu(B(x,r))}\, d\mu(y)\, d\mu(x)\right)^{1/2}
\end{equation*}
from~\cite[Theorem 2.4]{ABCRST3}. 
Notice that in the framework of~\cite{ABCRST3}, Ahlfors regularity is assumed, however the proof of~\cite[Theorem 2.4]{ABCRST3} can be generalized using the estimates \eqref{E:subGaussian_HKE} since they imply the doubling property of the measure. To conclude, it remains to prove that
\[
\mathbf{Var}_{2,\mathcal{E}} (f)  \simeq \liminf_{r\to 0^+} \frac{1}{r^{d_W/2}}\left( \int_X\int_{B(x,r)}\frac{|f(y)-f(x)|^2}{ \mu(B(x,r))}\, d\mu(y)\, d\mu(x)\right)^{1/2}.
\]
The lower bound is obtained in Theorem~\ref{W1p:I}, whereas the upper bound
\[
\mathbf{Var}_{2,\mathcal{E}} (f)  \leq c \liminf_{r\to 0^+} \frac{1}{r^{d_W/2}}\left( \int_X\int_{B(x,r)}\frac{|f(y)-f(x)|^2}{ \mu(B(x,r))}\, d\mu(y)\, d\mu(x)\right)^{1/2}
\]
follows from~\eqref{ref:op} with $p=2$, by letting $t \to 0^+$ and choosing $\delta$ large enough.
\end{proof}

\subsubsection*{The case $p=1$}
For completeness, and to include underlying spaces that are compact or negatively curved RCD spaces, we finish this section with a local version of the main results obtained in in~\cite{ABCRST3} and refer to the latter for further properties of BV functions in this setting.

\begin{theorem}\label{T:BV3}
Let $R\in (0,+\infty]$ and assume $(\mathrm G_\infty(R))$, i.e. there exists a constant $C>0$ such that for every $t \in (0,R)$, $f\in L^\infty(X,\mu)$ and $x,y \in X$,
\begin{equation}\label{E:wBECD}
| P_t f (x)-P_t f(y)| \le C \frac{d(x,y)^{d_W( 1-\alpha_1)}}{t^{1-\alpha_1}} \| f \|_{ L^\infty(X,\mu)}.
\end{equation}
Then, $(\mathrm {PPI}_1(R))$ is satisfied and the following equivalences of semi-norms is valid on $ BV(\mathcal E)$:
\begin{align*}
\mathbf{Var}_{\mathcal{E}} (f)   
& \simeq \liminf_{r\to 0^+} \int_X\int_{B(x,r)}\frac{|f(y)-f(x)|}{ r^{\alpha_1d_W}\mu(B(x,r))}\, d\mu(y)\, d\mu(x) \\
 & \simeq \sup_{r\in (0,R)} \int_X\int_{B(x,r)}\frac{|f(y)-f(x)|}{ r^{\alpha_1d_W}\mu(B(x,r))}\, d\mu(y)\, d\mu(x) \\
 &\simeq \| f \|_{1,\alpha_1,R}.
\end{align*}
\end{theorem}

\begin{proof}
By virtue of~\eqref{E:wBECD}, the local pseudo Poincar\'e inequality
\begin{equation*}
\| P_t f - f\|_{L^1(X,\mu)} \le C t^{\alpha_1} \mathbf{Var}_\mathcal{E} (f)
\end{equation*}
holds for any $t\in (0,R)$. Applying the latter as in~\cite[Lemma 4.12]{ABCRST3} yields
\begin{equation*}
\frac{1}{t^{\alpha_1}}\int_{X}P_t(|f-f(x)|(x)\,d\mu(x)\leq C\mathbf{Var}_\mathcal{E} (f)
\end{equation*}
and taking $\sup_{t\in (0,R)}$ on both sides we obtain 
$\| f \|_{1,\alpha_1,R}\simeq \mathbf{Var}_\mathcal{E} (f)$.
On the other hand, the lower heat kernel bound~\eqref{E:subGaussian_HKE} implies
\begin{equation*}
t^{-\alpha_1}\int_XP_t(|f-f(x)|)(x)\,d\mu(x)\geq c_1e^{c_2}\int_X\int_{B(x,t^{1/d_W})}\frac{|f(x)-f(y)|}{t^{\alpha_1}\mu(B(x,t^{1/d_W})}d\mu(x),
\end{equation*}
which setting $r=t^{1/d_W}$ reads
\begin{equation}\label{E:BV3_04}
t^{-\alpha_1}\int_XP_t(|f-f(x)|)(x)\,d\mu(x)\geq c_1e^{c_2}\int_X\int_{B(x,r)}\frac{|f(x)-f(y)|}{r^{\alpha_1 d_W}\mu(B(x,r))}d\mu(x).
\end{equation}
Taking $\liminf_{t\to 0^+}$ on both sides of the inequality, that is tantamount to taking $\liminf_{r\to 0^+}$ on the left hand side, we obtain 
\[
\mathbf{Var}_\mathcal{E} (f)\geq \liminf_{r\to 0^+}\int_X\int_{B(x,r)}\frac{|f(x)-f(y)|}{r^{\alpha_1 d_W}\mu(B(x,r))}d\mu(x).
\]
The converse inequality is proved in~\cite[Lemma 4.13]{ABCRST3} and requires only the heat kernel estimates~\eqref{E:subGaussian_HKE}, in particular no Bakry-\'Emery estimate, hence
\begin{equation}\label{E:BV3_05}
\mathbf{Var}_\mathcal{E} (f)\simeq \liminf_{r\to 0^+}\int_X\int_{B(x,r)}\frac{|f(x)-f(y)|}{r^{\alpha_1 d_W}\mu(B(x,r)})d\mu(x).
\end{equation}
Let us now consider $R\in (0,1)$. Taking $\sup_{r\in (0,R)}$ on both sides of~\eqref{E:BV3_04} while noticing that $r\in(0,R)$ implies $t=r^{d_W}\in (0,R^{d_W})\subset (0,R)$ we get
\begin{multline*}
\sup_{r\in (0,R)}\int_X\int_{B(x,r)}\frac{|f(x)-f(y)|}{r^{\alpha_1 d_W}\mu(B(x,r))}d\mu(x)
\leq C_1 \|f\|_{1,\alpha_1,R}\leq C_1 \mathbf{Var}_\mathcal{E} (f)\\
\leq C_2\sup_{r\in (0,R)}\int_X\int_{B(x,r)}\frac{|f(x)-f(y)|}{r^{\alpha_1 d_W}\mu(B(x,r))}d\mu(x),
\end{multline*}
where the last inequality follows from~\eqref{E:BV3_05} and the constants do not depend on $R$.
If $R\geq 1$,  we have that~\eqref{E:wBECD} holds for any $t>0$ and in particular for $t\in (0,R^{d_W})$. The first part of the present proof thus yields
$\| f \|_{1,\alpha_1,R^{d_W}}\simeq \mathbf{Var}_\mathcal{E} (f)$.
Taking $\sup_{r\in (0,R)}$ on both sides of~\eqref{E:BV3_04} we obtain in this case
\[
\sup_{r\in (0,R)}\int_X\int_{B(x,r)}\frac{|f(x)-f(y)|}{r^{\alpha_1 d_W}\mu(B(x,r))}d\mu(x)\leq C_1 \|f\|_{1,\alpha_1,R^{d_W}}
\]
and the remaining inequalities follows as in the previous case $R\in (0,1)$.
\end{proof}

\begin{remark}
Besides the comparison between $W^{1,p}(\mathcal E)$ and the Korevaar-Schoen spaces, it would be interesting to study their relation to Haj\l asz spaces~\cite{Haj96}. For $p \ge 1$ and $ \alpha \in (0,1]$, the Haj\l asz space $H^{p,\alpha}(X)$ on a metric measure space $(X,d,\mu)$ is defined as
\[
H^{p,\alpha}(X)=\left\{ f \in L^p(X,\mu), \, \exists g \in L^p(X,\mu), \, | f(x) -f(y) | \le d(x,y)^\alpha (g(x)+g(y)) \, \mu \, a.e. \right\}.
\]
From their definitions 
one can prove that if the heat kernel has sub-Gaussian estimates as in \eqref{E:subGaussian_HKE}, then 
\[
H^{p,\alpha}(X) \subset \mathbf{B}^{p,\frac{\alpha}{ d_W} }(X).
\]
The converse inclusion likely requires more assumptions on the underlying space $(X,d,\mu)$ which are related to curvature type lower bounds; in the context of RCD spaces, see~\cite{ABT18}.
\end{remark}

\subsection{Fractal spaces}\label{subS:fractals}

\subsubsection*{Nested fractals}
One class of fractals that are known to fit in the general framework of this paper are so-called nested fractals, among which the Sierpinski gasket is one of the most prominent examples. We refer to~\cite{Lin90,Fuk92,FHK94} for details on their general definition and the construction of a naturally associated diffusion process. In particular, nested fractals are also fractional metric spaces whose natural diffusion process is a fractional diffusion in the sense of Barlow~\cite[Definition 3.2]{Bar98}. The following theorem summarizes the results currently available that put these spaces into our setting; the proofs can be found in Theorem 3.7, Theorem 4.9 and Theorem 5.1 of~\cite{ABCRST3}. By an ``infinite'' fractal we mean its blow-up as introduced by R. S. Strichartz in~\cite{Str98}.

\begin{theorem}\label{T:nested_fractals}
Let $(X,d,\mu)$ be a compact or infinite nested fractal with $1\leq d_H\leq d_W$. Then, it satisfies $(\mathrm G_\infty)$. More precisely, the weak Bakry-\'Emery condition
\begin{equation}\label{E:BE_nested}
|P_tf(x)-P_tf(y)|\leq C\frac{d(x,y)^{d_W-d_H}}{t^{(d_W-d_H)/d_W}}\|f\|_{L^\infty(X,\mu)}\qquad t>0
\end{equation}
for some $C>0$ and any $f\in L^\infty (X,\mu)$ is satisfied. Moreover, $\alpha_1=d_H/d_W$ and
\begin{equation}\label{E:Var_nested}
 \|f\|_{1,d_H/d_W}\simeq \mathbf{Var}_{\mathcal{E}} (f).
\end{equation}
\end{theorem}
\begin{proof}
The statement for infinite nested fractals is fully proved in~\cite{ABCRST3}. In the case of compact nested fractals, using the \textit{local} estimates on the derivative of the heat kernel as in~\cite[Theorem 3.7]{ABCRST3} one obtains the weak Bakry-\'Emery condition \textit{locally} for $t$ in a bounded interval. By virtue of Lemma~\ref{L:local_wBECD2}, the condition extends to any $t>0$. The second statement is~\cite[Theorem 5.1]{ABCRST3}.
\end{proof}
\begin{remark}
The condition~\eqref{E:BE_nested} actually holds for a more general class of fractals, c.f.~\cite[Theorem 3.7]{ABCRST3}, however the statement concerning $\alpha_1$ and the equivalence of norms~\eqref{E:Var_nested} is so far only valid for nested fractals. It is conjectured in~\cite[Conjecture 5.4]{ABCRST3} that for fractals like the Sierpinski carpet one has $\alpha_1=(d_H-d_{tH}+1)/d_W$, where $d_{tH}$ denotes the topological Hausdorff dimension of the space.
\end{remark}

\subsection*{Vicsek set}

An interesting specific example within this class of nested fractals is the standard Vicsek set in $\mathbb{R}^2$ equipped with its corresponding Dirichlet form $(\mathcal{E},\mathcal{F})$, see e.g.~\cite[p.26]{Bar98}. This is a fractional space with a fractional diffusion in the sense of Barlow and we know e.g. from Theorem~\ref{T:nested_fractals} that $\alpha_1=\frac{d_H}{d_W}$. In fact, it is possible to explicitly construct non-constant functions $h\in\mathcal{F}$ that belongs to $\mathbf{B}^{p,\beta_p}(X)$ for any $p\geq 1$ and $\beta_p=\big( 1-\frac{2}{p}\big)(1-\alpha_1)+\frac{1}{p}$ as in Remark~\ref{R:wBE_to_Gq}. We shall see that such a function $h$ (whose construction is inspired by~\cite[ Theorem 5.2]{ABCRST3}) is in fact a \textit{harmonic} function, and the construction may be generalized to so-called $m$-harmonic functions. 

\medskip

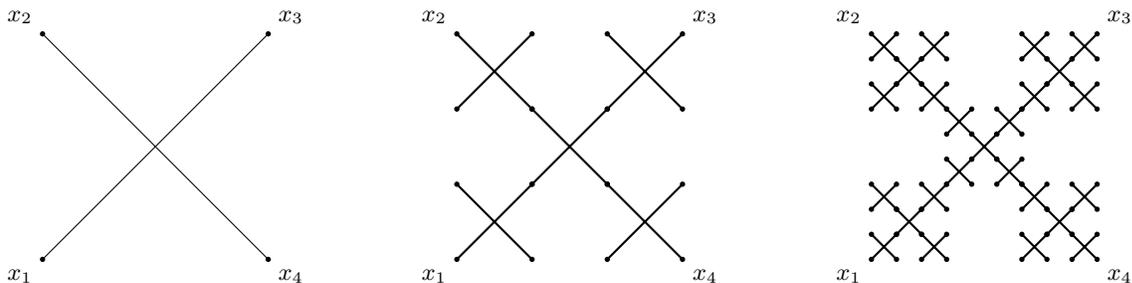
\begin{figure}[H]
\centering
\begin{tabular}{ccc}
\begin{tikzpicture}[scale=3]
\foreach \a/\b in {0/0}{
\coordinate[label=below left:{\footnotesize $x_1$}] (x1) at (\a,\b) ;
\fill (\a,\b) circle (1/3pt);
\coordinate[label=above left:{\footnotesize $x_2$}] (x2) at (\a,\b+1) ;
\fill (\a,\b+1) circle (1/3pt);
\coordinate[label=above right:{\footnotesize $x_3$}] (x3) at (\a+1,\b+1) ;
\fill (\a+1,\b+1) circle (1/3pt);
\coordinate[label=below right:{\footnotesize $x_4$}] (x4) at (\a+1,\b) ;
\fill (\a+1,\b) circle (1/3pt);
\draw (\a,\b) -- (\a+1,\b+1) (\a,\b+1) -- (\a+1,\b);
}
\end{tikzpicture}
\hspace*{2em}
&
\begin{tikzpicture}
\coordinate[label=below left:{\footnotesize $x_1$}] (x1) at (0,0) ;
\coordinate[label=above left:{\footnotesize $x_2$}] (x2) at (0,3) ;
\coordinate[label=above right:{\footnotesize $x_3$}] (x3) at (3,3) ;
\coordinate[label=below right:{\footnotesize $x_4$}] (x4) at (3,0) ;
\foreach \c/\d in {0/0}{
\foreach \a/\b in {0/0, 0/2, 1/1, 2/0, 2/2}{
\fill[shift={(\c,\d)}] (\a,\b) circle (1pt);
\fill[shift={(\c,\d)}] (\a,\b+1) circle (1pt);
\fill[shift={(\c,\d)}] (\a+1,\b+1) circle (1pt);
\fill[shift={(\c,\d)}] (\a+1,\b) circle (1pt);
\draw[thick,shift={(\c,\d)}] (\a,\b) -- (\a+1,\b+1) (\a,\b+1) -- (\a+1,\b);
}
}
\end{tikzpicture}
\hspace*{2em}
&
\begin{tikzpicture}[scale=1/3]
\coordinate[label=below left:{\footnotesize $x_1$}] (x1) at (0,0) ;
\coordinate[label=above left:{\footnotesize $x_2$}] (x2) at (0,9) ;
\coordinate[label=above right:{\footnotesize $x_3$}] (x3) at (9,9) ;
\coordinate[label=below right:{\footnotesize $x_4$}] (x4) at (9,0) ;
\foreach \c/\d in {0/0, 0/6, 3/3, 6/0, 6/6}{
\foreach \a/\b in {0/0, 0/2, 1/1, 2/0, 2/2}{
\fill[shift={(\c,\d)}] (\a,\b) circle (3pt);
\fill[shift={(\c,\d)}] (\a,\b+1) circle (3pt);
\fill[shift={(\c,\d)}] (\a+1,\b+1) circle (3pt);
\fill[shift={(\c,\d)}] (\a+1,\b) circle (3pt);
\draw[thick,shift={(\c,\d)}] (\a,\b) -- (\a+1,\b+1) (\a,\b+1) -- (\a+1,\b);
}
}
\end{tikzpicture}
\end{tabular}
\caption{Approximating graphs $(V_m,E_m)$ for the Vicsek set.}
\label{F:VicsekSet}
\end{figure}

Denote by $\{\psi_i\}_{i=1}^5$ the contraction mappings that generate $X$ and define for any $w\in\{1,\ldots, 5\}^m$ the mapping $\psi_w:=\psi_{w_1}{\circ}\ldots\circ\psi_{w_m}$ that generates an $m$-level copy of $X$, so that $X=\bigcup_{w\in\{1,\ldots,5\}^{m}}\psi_w(X)$. One can approximate $X$ by a sequence of graphs $\{(V_m,E_m)\}_{m\geq 0}$ as illustrated in Figure~\ref{F:VicsekSet}. 
A function $h\colon X\to\mathbb{R}$ is said to be $m$-harmonic if it arises as the energy minimizing extension of a given function with values on the approximation level $m$, i.e.
\[
\mathcal{E}(h,h)=\inf\{\mathcal{E}(g,g)\,\colon\,g|_{V_m}=f_m\}
\]
for some $f_m\colon V_m\to\mathbb{R}$. Following the notation and the result in~\cite[Proposition 7.13]{Bar98}, we write in this case $h=H_mf_m$ and know that $H_mf_m\in\mathcal{D}\cap C(X)$.  

\begin{theorem}\label{T:Vicsek1}
On the Vicsek set, the space $\mathbf{B}^{2,1/2}(X)\cap\mathbf{B}^{p,\beta_p}(X)$ contains non-trivial functions for any $p\geq 1$. In particular, for $1 \le p \le 2$,
\[
\alpha_p=\left( 1-\frac{2}{p}\right)\left(1-\frac{d_H}{d_W}\right)+\frac{1}{p}
\]
and $(\mathrm {PPI}_p)$ is satisfied.
\end{theorem}

\begin{proof}
Let us consider graph approximation $(V_0,E_0)$ and a function $f_0\colon V_0\to\mathbb{R}$ that takes the values $a_1,a_2,a_3,a_4$ on each vertex $x_1,x_2,x_3,x_4$ of $V_0$, respectively. For simplicity, we assume that the function is only non-zero at two connected vertices, say $x_1$ and $x_3$. A generic function in $V_0$ can be analyzed by writing is as the sum of two that are zero at complementary pairs of (connected) vertices. The harmonic extension of $f_0$ to the Vicsek set $X$ is defined as the function $h:=H_0f_0\in\mathcal{F}$ such that $h|_{V_0}\equiv f_0$ and
\[
\mathcal{E}(h,h)=\min\{\mathcal{E}(f,f)\,\colon\,f\in\mathcal{F}\text{ and }f_{|_{V_0}}=f_0\}.
\]
This $0$-harmonic function $h$ is obtained by linear interpolation on the diagonal that joins $x_1$ and the upper-right corner $x_3$. We call this the ``distinguished'' diagonal. On all branches intersecting it, including the other diagonal crossing lower-right to upper-left, $h$ is constant according to its value on the distinguished diagonal. 
This harmonic extension is clearly non-constant, 
it is unique and belongs to $\mathcal{F}=\mathbf{B}^{2,1/2}(X)$, see e.g.~\cite[Lemma 8.2]{Kig12}. In order to prove that $\|h\|_{p,\beta_p}<\infty$ for any $p\ge 1$, we first fix $r\in (0,1/6)$ and set $n:=n_r\geq 0$ to be the largest such that $2r<3^{-(n+1)}$. 
Note that $X$ can be covered by $5^n$ squares of side length $3^{-n}$, which we denote $\{Q_i^{(n)}\}_{i=1}^{5^n}$. By construction, the function $h$ is constant on cells $B_i^{(n)}:=X\cap Q_i^{(n)}$ for which $Q_i^{(n)}$ does not intersect the distinguished diagonal of $X$. In addition, $h$ is also constant on the $r$-neighborhood of any such cell, i,e.
\begin{equation*}
|h(x)-h(y)|=0\quad\text{for any }y\in B_i^{(n)}\text{ and } x\in B(y,r).
\end{equation*}
In other words, among the $n$-cells $\{B_i^{(n)}\}_{i=1}^{5^{n}}$, only in $3^n$ of them the latter difference is nonzero. Since $h$ is by definition linear, on any of these $3^n$ cells it holds that
\begin{equation*}
|h(x)-h(y)|\leq d(x,y)\quad\text{for all }y\in B_i^{(n)}\text{ and } x\in B(y,r).
\end{equation*}
Combining this two facts and using the Ahlfors regularity of the space we have for any $p\geq 1$
\begin{align*}
&\frac{1}{r^{p\alpha_pd_W+d_H}}\int_X\int_{B(y,r)}|h(x)-h(y)|^pd\mu(x)\,d\mu(y)\\
&\leq \frac{1}{r^{p\beta_pd_W+d_H}}\sum_{i=1}^{3^n}\int_{B_i}\int_{B(y,r)}r^p\,d\mu(x)\,d\mu(y)\\
&\leq\frac{C}{r^{p\beta_pd_W+d_H}}\sum_{i=1}^{3^n}r^{p+d_H}\mu(B_i) 
\leq \frac{C}{r^{p\beta_pd_W+d_H}}3^nr^{p+d_H}\Big(\frac{3r}{2}\Big)^{d_H}\\
&\leq \frac{C}{r^{p\beta_pd_W+d_H}} r^{-1+p+2d_H}
=C r^{p+d_H-(1+p\beta_pd_W)}.
\end{align*}
From Theorem~\ref{T:nested_fractals} we know that $\beta_1=\frac{d_H}{d_W}$, which substituting above yields the exponent 
\begin{equation}\label{E:Vicsek_01}
p+d_H-(1+p\beta_pd_W)=(p-1)(1+d_H-d_W).
\end{equation}
In addition, the Vicsek set satisfies $d_W=1+d_H$, c.f.~\cite[Theorem 8.18]{Bar98}, hence~\eqref{E:Vicsek_01} equals zero and we get
\begin{equation*}
\frac{1}{r^{p\beta_pd_W+d_H}}\int_X\int_{B(y,r)}|h(x)-h(y)|^pd\mu(x)\,d\mu(y)\leq C.
\end{equation*}
Since the bound is independent of $r$, we may now estimate
\[
\sup_{r\in (0,1/6)}\frac{1}{r^{p\beta_pd_W+d_H}}\int_X\int_{B(y,r)}|h(x)-h(y)|^pd\mu(x)\,d\mu(y)\leq C 
\]
which in view of~\cite[Theorem 2.4]{ABCRST3} yields
\begin{equation*}
\|h\|_{p,\beta_p}\leq C_{p,\beta_p}(C+6^{\beta_pd_W}\|h\|_{L^p(X,\mu)})
\end{equation*}
as we wanted to prove.

The space $\mathbf{B}^{p,\beta_p}(X)$ is therefore non trivial. By definition of the critical exponent $\alpha_p$ this yields $ \alpha_p \ge \beta_p$ and~\cite[Theorem 3.11]{ABCRST3} yields $ \alpha_p = \beta_p$. Finally one obtains the property $(\mathrm {PPI}_p)$ from~\cite[Theorem 3.10]{ABCRST3}.

\end{proof}

\begin{remark}\label{R:Vicsek2}
It is actually possible to prove that any $m$-harmonic function $H_mf$ on the Vicsek set belongs to $\mathbf{B}^{p,\beta_p}(X)$ for any $p\geq 1$ and that there exists $C>0$ independent of $m$ such that
\begin{equation}\label{E:Vicsekmharmonic}
\|H_mf_{|_{W_m}}\|_{p,\beta_p}\leq C\|f\|_{L^\infty(X,\mu)}.
\end{equation}
As a consequence, one can deduce that $\alpha_p=\beta_p$ for every $p \ge 2$. For concision, the proof of this fact is postponed to~\cite{ABCRST5}. We note that however, the question of the validity of $(\mathrm {PPI}_p)$ for $p>2$ is still open.
\end{remark}

\subsubsection*{Products of nested fractals}
Higher dimensional examples of fractal spaces can be constructed by taking products; we refer to~\cite{Str04} for further details and results regarding heat kernel estimates on such fractals. In particular, as noticed in~\cite[Section 3.3]{ABCRST3}, given a nested fractal $X$ that satisfies the sub-Gaussian estimates~\eqref{E:subGaussian_HKE}, is $n$-fold product $X^n$ will have Hausdorff dimension $nd_H$, while its walk dimension $d_W$ remains unchanged. The next theorem puts these spaces into our setting. 

\begin{theorem}[Proposition 3.8, Theorem 5.6\cite{ABCRST3}]\label{T:Vicsek_prod}
Let $(X,d,\mu)$ be a nested fractal with $1\leq d_H\leq d_W$. Then, Theorem~\ref{T:nested_fractals} holds with the same exponents for any $n$-fold product $(X^n,d_{X^n},\mu^{\otimes n})$, $n\geq 1$.
\end{theorem}

In the case of the Vicsek set, and in view of Theorem~\ref{T:Vicsek1} and Remark~\ref{R:Vicsek2} one has the following result.

\begin{theorem}
Let $(X,d,\mu)$ denote the Vicsek set. For the $n$-fold product $(X^n,d_{X^n},\mu^{\otimes n})$, $n\geq 1$, for any $p \ge 1$ it holds that
\[
\alpha_p=\left( 1-\frac{2}{p}\right)\left(1-\frac{d_H}{d_W}\right)+\frac{1}{p}
\]
and $(\mathrm {PPI}_p)$ is satisfied for any $1 \le p \le 2$, where $d_H$ is the Hausdorff dimension of $X$ and $d_W$ the walk dimension of $X$.
\end{theorem}

\section{Gagliardo-Nirenberg and Trudinger-Moser inequalities}\label{section GN-TM}

We now turn to the core of the paper and show how the pseudo-Poincar\'e inequalities introduced in Section~\ref{subS:PPIs} can be applied to obtain the whole range of Gagliardo-Nirenberg and Trudinger-Moser inequalities for the Sobolev spaces $W^{1,p}(\mathcal{E})$. The techniques used rely on Lemma~\ref{L:local_norm_approx:a} in conjunction with general methods developed in~\cite{BCLS95}.

\subsection{Global versions}

We start by recalling once again that, since $(X,\mu)$ is assumed to be a good measurable space, the semigroup $\{P_t\}_{t\geq 0}$ associated with the Dirichlet form $(\mathcal{E},\mathcal{F})$ admits a measurable heat kernel $p_t(x,y)$, c.f.~\cite[Theorem 1.2.3]{BGL14}). Throughout this section we will assume that the heat kernel satisfies
\begin{equation}\label{E:subGauss-upper}
p_{t}(x,y)\leq C_h t^{-\beta}
\end{equation}
for some $C_h>0$ and $\beta >0$, $\mu\times\mu$-a.e.\ $(x,y)\in X\times X$ and any $t>0$. In addition, we will consider for each $p\ge 1$ the $L^p$ pseudo-Poincar\'e inequality $(\rm{PPI}_p)$ from Section~\ref{subS:PPIs}: There exists a constant $C_p>0$ such that for every $t \ge 0$ and $f \in W^{1,p}(\mathcal{E})$ (or $BV(\mathcal{E})$ for $p=1$),
\[
\| P_t f -f \|_{L^p(X,\mu)} \le C_p t^{\alpha_p} \mathbf{Var}_{p,\mathcal{E}} (f).
\]

The following result extends to the abstract Dirichlet space framework the classical Gagliardo-Nirenberg inequalities, see e.g.~\cite{Bad09}.

\begin{theorem}\label{T:Gagliardo-Nirenberg}
Assume  that $(\mathrm {PPI}_p)$ is satisfied for some $p \ge 1$. Then, there exists a constant $c_{p}>0$ such that for every $f \in W^{1,p}(\mathcal E)$ (or $BV(\mathcal{E})$ for $p=1$),
\begin{equation}\label{E:Gagliardo-Nirenberg}
\| f \|_{L^q(X,\mu)} \le c_pC_{p}^{\frac{\beta}{\beta+\alpha_p}}C_h^{\frac{\alpha_p}{\beta+\alpha_p}} \mathbf{Var}_{p,\mathcal{E}} (f)^{\frac{\beta}{\beta+\alpha_p}} \| f \|^{\frac{\alpha_p}{\beta+\alpha_p}}_{L^1(X,\mu)},
\end{equation}
where $q=p \big( 1 +\frac{\alpha_p}{\beta} \big)$.
\end{theorem}


\begin{proof}
For $p\geq 1$, we set $\theta:=p/q$ and consider the semi-norm
\begin{equation}\label{E:Ledoux_norm}
\|f\|_{B^{\alpha_p\theta/(\theta-1)}_{\infty,\infty}}=\sup_{t >0} t^{-\alpha_p\theta/(\theta-1)} \| P_t f \|_{L^\infty(X,\mu)}.
\end{equation}
Let $f\in W^{1,p}(\mathcal E)$ (or $BV(\mathcal E)$ if $p=1$) and assume first that $f\geq 0$ and also that, by homogeneity, $\|f\|_{B^{\alpha\theta/(\theta-1)}_{\infty,\infty}}\leq 1$. 
For any $s>0$, set $t_s=s^{\frac{\theta-1}{\alpha\theta}}$ so that $|P_{t_s}f|\leq s$.
Then,
\begin{multline}\label{ref:1}
s^q\mu\big(\{x\in X\;\colon\;|f(x)|\geq 2s\}\big)\leq s^q\mu\big(\{x\in X\;\colon\;|f-P_{t_s}f|\geq s\}\big)\\
\leq 
s^{q-p}\|f-P_{t_s}f\|_{L^p(X,\mu)}^p\leq s^{q-p}t_s^{p\alpha}C_p^p\mathbf{Var}_{p,\mathcal{E}} (f)^p=C_p^p\mathbf{Var}_{p,\mathcal{E}} (f)^p,
\end{multline}
where the last inequality follows from $({\rm PPI}_p)$ and the last equality from $q-p+p(\theta-1)/\theta=0$. 
Let us now define $f_k:=\min\{(f-2^k)_+,2^k\}$, $k\in\mathbb{Z}$. We note that $0\le f_k \le f $, so that
\[
\|f_k\|_{B^{\alpha\theta/(\theta-1)}_{\infty,\infty}} \le \|f\|_{B^{\alpha\theta/(\theta-1)}_{\infty,\infty}} \le 1.
\]

Applying \eqref{ref:1} to $f_k$ with $s=2^k$ yields
\[
2^{kq}\mu\big(\{x\in X\;\colon\;|f_k(x)|\geq 2^{k+1}\}\big)\leq C_p^p\mathbf{Var}_{p,\mathcal{E}} (f_k)^p
\]
so that from Lemma \ref{L:local_norm_approx:a} we deduce
\[
\sum_{k\in\mathbb{Z}}2^{kq}\mu\big(\{x\in X\;\colon\;|f_k(x)|\geq 2^{k+1}\}\big)\leq C_p^p\sum_{k\in\mathbb{Z}}\mathbf{Var}_{p,\mathcal{E}} (f_k)^p\leq 2^p(p+1)^pC_p^p\mathbf{Var}_{p,\mathcal{E}} (f)^p.
\]
Further,
\begin{align*}
\|f\|_{L^q(X,\mu)}^q&=\int_0^\infty q s^{q-1}\mu\big(\{x\in X\,\colon\,|f(x)| \ge s\}\big)\,ds\\
&= \sum_{k\in\mathbb{Z}}\int_{2^{k+1}}^{2^{k+2}}q s^{q-1}\mu\big(\{x\in X\,\colon\,|f(x)| \ge s\}\big)\,ds\\
&\leq \sum_{k\in\mathbb{Z}}\int_{2^{k+1}}^{2^{k+2}}q s^{q-1}\mu\big(\{x\in X\,\colon\,|f(x)| \ge 2^{k+1}\}\big)\,ds\\
&=(2^{2q}-2^q)\sum_{k\in\mathbb{Z}}2^{kq}\mu\big(\{x\in X\,\colon\,|f(x)| \ge 2^{k+1}\}\big)\\
&\leq (2^{2q}-2^q)\sum_{k\in\mathbb{Z}}2^{kq}\mu\big(\{x\in X\,\colon\,|f_k(x)| \ge 2^{k}\}\big)\\
&=(2^{3q}-2^{2q})\sum_{k\in\mathbb{Z}}2^{kq}\mu\big(\{x\in X\,\colon\,|f_k(x)| \ge 2^{k+1}\}\big)
\leq 2^{3q} 2^p(p+1)^pC_p^p\mathbf{Var}_{p,\mathcal{E}} (f)^p.
\end{align*}
One concludes that for every $f\in W^{1,p}(\mathcal E)$ (or $BV(\mathcal E)$ if $p=1$) such that  $f\geq 0$
\begin{equation}\label{E:Gagliardo-Nirenberg_01}
\| f \|_{L^q(X,\mu)} \le 2^32^\theta(p+1)^\theta C_p^\theta\mathbf{Var}_{p,\mathcal{E}} (f)^\theta \| f \|^{1-\theta}_{B_{\infty,\infty}^{\alpha_p \theta/(\theta-1)}},
\end{equation}
where $\theta=\frac{p}{q}$. On the other hand, the heat kernel upper bound $p_{t}(x,y)\leq C_h t^{-\beta}$ implies
\[
\| P_t f \|_{L^\infty(X,\mu)}\le \frac{C_h}{t^{\beta}} \|  f \|_{L^1(X,\mu)}
\]
and by definition, see~\eqref{E:Ledoux_norm}, it follows from~\eqref{E:Gagliardo-Nirenberg_01} that
\[
\| f \|_{L^q(X,\mu)} \le 2^32^\theta(p+1)^\theta C_p^\theta C_h^{1-\theta} \mathbf{Var}_{p,\mathcal{E}} (f)^\theta \|f\|_{L^1(X,\mu)}^{1-\theta}
\]
for $\beta= \frac{\alpha \theta}{1-\theta}=\frac{\alpha_p p}{q-p}$, equivalently $\frac{1}{q}=\frac{1}{p}-\frac{\alpha_p}{q\beta}$, as we wanted to prove. If one does not assume $f \ge 0$, then the previous inequality applied to $| f |$ yields the expected result, since it is clear from the definition that $\mathbf{Var}_{p,\mathcal{E}} (| f |) \le \mathbf{Var}_{p,\mathcal{E}} ( f )$.
\end{proof}

\subsubsection{Gagliardo-Nirenberg}
Thanks to general results proved in~\cite{BCLS95}, Theorem~\ref{T:Gagliardo-Nirenberg} actually implies the full scale of Gagliardo-Nirenberg inequalities. We discuss them according to the value of $p \alpha_p$.

\begin{corollary}\label{C:Gagliardo-Nirenberg:corollary1}
Assume that $(\mathrm {PPI}_p)$ is satisfied for some $p \ge 1 $ such that $p \alpha_p <\beta$. Then, there exists a constant $C_{p,r,s}>0$ such that for every $f \in W^{1,p}(\mathcal E)$ (or $BV(\mathcal{E})$ for $p=1$),
\begin{equation}\label{E:Gagliardo-Nirenberg:corollary1}
\| f \|_{L^r(X,\mu)} \le C_{p,r,s} \mathbf{Var}_{p,\mathcal{E}} (f)^{\theta} \| f \|^{1-\theta}_{L^s(X,\mu)},
\end{equation}
where $r,s \in [1,+\infty]$ and $\theta \in (0,1]$ are related by the identity
\[
\frac{1}{r}=\theta \Big(\frac{1}{p}-\frac{\alpha_p}{\beta} \Big)+\frac{1-\theta}{s}.
\]
\end{corollary}

\begin{proof}
This follows from Theorem~\ref{T:Gagliardo-Nirenberg} and~\cite[Theorem 3.1]{BCLS95}.
\end{proof}

\begin{remark}
Several special cases are worth pointing out explicitly:
\begin{enumerate}[leftmargin=1.5em,label={\rm (\roman*)}]
\item If $r=s$, then $r=\frac{p\beta}{\beta-p \alpha_p}$ and~\eqref{E:Gagliardo-Nirenberg:corollary1} yields the global Sobolev inequality
\[
\| f \|_{L^r(X,\mu)} \le C_{p} \mathbf{Var}_{p,\mathcal{E}} (f)
\]
\item If $r=p>1$ and $s=1$, then~\eqref{E:Gagliardo-Nirenberg:corollary1} yields the global Nash inequality
\[
\| f \|_{L^p(X,\mu)} \le C_{p} \mathbf{Var}_{p,\mathcal{E}} (f)^{\theta} \| f \|^{1-\theta}_{L^1(X,\mu)}
\]
with $\theta=\frac{(p-1)\beta}{p(\alpha_p+\beta)-\beta}$.
\item If $s=+\infty$, then~\eqref{E:Gagliardo-Nirenberg:corollary1} yields
\[
\| f \|_{L^r(X,\mu)} \le C_{p,r} \mathbf{Var}_{p,\mathcal{E}} (f)^{\theta} \| f \|^{1-\theta}_{L^\infty(X,\mu)}
\]
with $\theta=\frac{p\beta}{r(\beta-p \alpha_p)}$
\end{enumerate}
\end{remark}

We now turn to the case $p \alpha_p >\beta$.

\begin{corollary}\label{C:Gagliardo-Nirenberg:corollary2}
Assume that $(\mathrm {PPI}_p)$ is satisfied for some $p \ge 1 $ such that $p \alpha_p  > \beta$. Then, there exists a constant $C_p>0$ such that for every $f \in W^{1,p}(\mathcal E)$ (or $BV(\mathcal{E})$ for $p=1$), and $s \ge 1$,
\begin{equation}\label{E:Gagliardo-Nirenberg:corollary2}
\| f \|_{L^\infty(X,\mu)} \le C_{p} \mathbf{Var}_{p,\mathcal{E}} (f)^{\theta} \| f \|^{1-\theta}_{L^s(X,\mu)},
\end{equation}
where $\theta \in (0,1)$ is given by $\theta=\frac{p\beta}{p\beta+s(p\alpha_p-\beta)}$.
\end{corollary}

\begin{proof}
This follows from Theorem~\ref{T:Gagliardo-Nirenberg} and~\cite[Theorem 3.2]{BCLS95}.
\end{proof}

\begin{remark}\label{R:Gagliardo-Nirenberg:corollary2}
For $s=1$, we have that
\[
\| f \|_{L^s(X,\mu)}=\| f \|_{L^1(X,\mu)} \le \| f \|_{L^\infty(X,\mu)} \mu ( \mathrm{Supp} (f)),
\]
where $\mathrm{Supp} (f)$ denotes the support of $f$. Thus,~\eqref{E:Gagliardo-Nirenberg:corollary2} yields for any $f\in W^{1,p}(\mathcal{E})$ (or $BV(\mathcal{E})$)
\[
\| f \|_{L^\infty(X,\mu)} \le C_{p} \mathbf{Var}_{p,\mathcal{E}} (f) \mu ( \mathrm{Supp} (f))^{\frac{\alpha_p}{\beta}-\frac{1}{p}}.
\]
\end{remark}

\subsubsection{Trudinger-Moser}
The case $p \alpha_p =\beta$ corresponds to Trudinger-Moser inequalities. We start with the case $p=1$ that is particularly well-suited for applications to fractal spaces.

\begin{corollary}\label{C:Trudinger-Moser}
Assume that $(\mathrm {PPI}_1)$ is satisfied and that $\alpha_1=\beta$. Then, there exists a constant $C>0$ such that for every $ f \in BV(\mathcal{E})$:
\begin{equation*}
\| f \|_{L^\infty(X,\mu)} \le C \mathbf{Var}_{\mathcal{E}} (f).
\end{equation*}
\end{corollary}

\begin{proof}
By virtue of Lemma~\ref{L:local_norm_approx:a}, the condition $(H_1)$ from~\cite[Section 2]{BCLS95} is satisfied, hence Theorem~\ref{T:Gagliardo-Nirenberg} and~\cite[Theorem 3.2]{BCLS95} yield the result. 
\end{proof}

We finally conclude with the Trudinger-Moser inequalities corresponding to $p>1$.

\begin{corollary}\label{C:Trudinger-Moser_1}
Assume further that $(\mathrm {PPI}_p)$ is satisfied and that $p\alpha_p=\beta$ with $p>1$. Then, there exist constants $c,C>0$ such that 
\begin{equation*}
\int_X \left( e^{ c |f|^{\frac{p}{p-1}} }-1 \right) d\mu  \le C \| f \|_{L^1(X,\mu)}
\end{equation*}
holds for every $ f \in W^{1,p}(\mathcal{E})$ with $\mathbf{Var}_{p,\mathcal{E}} (f)=1$.
\end{corollary}

\begin{proof}
Once again, Lemma~\ref{L:local_norm_approx:a} implies condition $(H_p)$ from~\cite[Section 2]{BCLS95} for $p>1$, and the result follows from Theorem~\ref{T:Gagliardo-Nirenberg} and~\cite[Theorem 3.4]{BCLS95}.
\end{proof}

\subsection{Localized versions}
In order to be able to treat spaces that lack global estimates, as for instance hyperbolic spaces, $RCD(K,+\infty)$ spaces with $K<0$, or compact spaces where only the local time behavior is meaningful, in this section we adapt the previous ideas to obtain a local version of Theorem~\ref{T:Gagliardo-Nirenberg}. 
In the spirit of~\cite[Section 3.3.2]{SC02}, Theorem~\ref{T:Gagliardo-Nirenberg_loc} in fact provides a local inequality depending on a parameter $R$, which in the limit $R\to\infty$ recovers its global counterpart~\eqref{E:Gagliardo-Nirenberg}. 

\medskip



The local version of the property $({\rm PPI}_p)$ was introduced in Section~\ref{subS:PPIs} with the notation $(\mathrm {PPI}_p (R))$ for $p \ge 1$ and $R>0$ as follows: There exists a constant $C_p(R)>0$ such that for every $f \in W^{1,p}(\mathcal{E})$ (or $BV(\mathcal{E})$ for $p=1$),
\[
\| P_t f -f \|_{L^p(X,\mu)} \le C_p(R) t^{\alpha_p} \mathbf{Var}_{p,\mathcal{E}} (f)
\]
holds for every $t \in (0,R)$.
\begin{theorem}\label{T:Gagliardo-Nirenberg_loc}
Fix $R>0$, $p\geq 1$ and $\alpha>0$. Assume that the space $(X,d,\mu,\mathcal{E},\mathcal{F})$ satisfies:
\begin{enumerate}[leftmargin=2em,label={\rm (\roman*)}]
\item  The heat semigroup $P_t$ admits a measurable heat kernel $p_t(x,y)$ such that for some $C_h>0$ and $\beta >0$,
\begin{equation}\label{E:G-N-loc_UHKE}
p_{t}(x,y)\leq C_h t^{-\beta}
\end{equation} 
for $\mu\times\mu$-a.e.\ $(x,y)\in X\times X$ and each $0<t\leq R$;
\item The property $(\mathrm {PPI}_p (R))$, with constant $C_p(R)>0$. 
\end{enumerate}
Then, there exist $C_{p}>0$ such that for every $f \in L^p(X,\mu)$,
\begin{equation*}
  \| f \|_{L^q(X,\mu)} \le 4p(2p+2)^{\frac{\beta}{\beta+\alpha_p}}C_h^{\frac{\alpha_p}{\beta+\alpha_p}}\big(R^{-\alpha_p}\|f\|_{L^p(X,\mu)}+C_p(R) \mathbf{Var}_{p,\mathcal{E}} (f)\Big)^{\frac{\beta}{\beta+\alpha_p}}\|f\|_{L^1(X,\mu)}^{\frac{\alpha_p}{\beta+\alpha_p}},
\end{equation*}
where $\frac{1}{q}=\frac{1}{p}-\frac{ \alpha_p}{q \beta}$.
\end{theorem}

\begin{proof}
The following is a modification of the arguments used in Theorem~\ref{T:Gagliardo-Nirenberg}. 
With $\theta:=\frac{p}{q}\in(0,1)$, we consider the localized semi-norm

\begin{equation}\label{E:Ledoux_norm2}
\|f\|_{B^{\alpha\theta/(\theta-1)}_{R,\infty,\infty}}=\sup_{t\in (0,R)} t^{-\alpha\theta/(\theta-1)} \| P_t f \|_{L^\infty(X,\mu)}.
\end{equation}
Let $f \in L^p(X,\mu)$ and assume $f\geq 0$. By homogeneity, we consider $\|f\|_{B^{\alpha\theta/(\theta-1)}_{R,\infty,\infty}}\leq 1$.

Let now $s>0$. If $s>R^{\frac{\alpha_p\theta}{\theta-1}}=(1/R)^{\frac{\alpha_p\theta}{1-\theta}}$, we take $t=t_s:=s^{\frac{\theta_p-1}{\alpha\theta}}=(1/s)^{\frac{1-\theta}{\alpha_p\theta}}<R$, so that $|P_{t_s} f|<s$. By virtue of the property $(\mathrm {PPI}_p (R))$,
\begin{align*}
s^q\mu\big(\{x\in X\,\colon\,|f(x)|\geq 2s\}\big)&\leq s^{q-p}\|f-P_{t_s}f\|_{L^p(X,\mu)}^p\\
&\leq s^{q-p}t_s^{-\alpha p}\mathbf{Var}_{p,\mathcal{E}} (f)^p=C_p(R)^p\mathbf{Var}_{p,\mathcal{E}} (f)^p.
\end{align*}
Thus, for any $k\geq k_0$ with $2^{k_0-1}<R^{\frac{\alpha\theta}{\theta-1}}\leq 2^{k_0}$ and $f_k:=(f-2^k)_+ \wedge 2^k$, 
\begin{equation*}
2^{kq}\mu\big(\{x\in X\,\colon\,|f_k(x)|\geq 2^{k+1}\}\big)\leq C_p(R)^p\mathbf{Var}_{p,\mathcal{E}} (f_k)^p
\end{equation*}
and hence Lemma~\ref{L:local_norm_approx:a} yields
\begin{equation*}
\sum_{k=k_0}^\infty 2^{kq}\mu\big(\{x\in X\,\colon\,|f_k(x)|\geq 2^{k+1}\}\big)\leq  C_p(R)^p\sum_{k\in\mathbb{Z}} \mathbf{Var}_{p,\mathcal{E}} (f_k)^p\leq  C_p(R)^p2^p(p+1)^p\mathbf{Var}_{p,\mathcal{E}} (f)^p.
\end{equation*}
If $s< 2^{k_0}$, we write
\begin{equation*}
s^q\mu\big(\{x\in X\,\colon\,|f(x)|> s\}\big)\leq s^{q-p}\|f\|_{L^p(X,\mu)}^p.
\end{equation*}
Using the previous two estimates, and setting $k_0>0$ to be such that $2^{k_0-1}<R^{\frac{\alpha\theta}{\theta-1}}\leq 2^{k_0}$, we obtain
\begin{align*}
\|f\|_{L^q(X\,u)}^q&=\int_0^\infty qs^{q-1}\mu\big(\{x\in X\,\colon\,|f(x)|>s\}\big)\,ds\\
&=\int_0^{2^{k_0+1}}qs^{q-1}\mu\big(\{x\in X\,\colon\,|f(x)|>s\}\big)\,ds+\int_{2^{k_0+1}}^\infty qs^{q-1}\mu\big(\{x\in X\,\colon\,|f(x)|>s\}\big)\,ds\\
&\leq \|f\|_{L^p(X,\mu)}^p\int_0^{2^{k_0+1}}qs^{q-p-1}ds+\sum_{k=k_0}^\infty\int_{2^{k+1}}^{2^{k+2}}qs^{q-1}\mu\big(\{x\in X\,\colon\,|f(x)|>s\}\big)\,ds\\
&\leq \|f\|_{L^p(X,\mu)}^p\frac{q2^{(k_0+1)(q-p)}}{q-p}+\sum_{k=k_0}^\infty\int_{2^{k+1}}^{2^{k+2}}qs^{q-1}\mu\big(\{x\in X\,\colon\,|f(x)|>2^{k+1}\}\big)\,ds\\
&\leq \|f\|_{L^p(X,\mu)}^p\frac{q4^{q-p}}{q-p}R^{\frac{\alpha\theta(q-p)}{\theta-1}}+2^q(2^q-1)\sum_{k=k_0}^\infty {2^{qk}}\mu\big(\{x\in X\,\colon\,|f(x)|>2^{k+1}\}\big)\\
&\leq \|f\|_{L^p(X,\mu)}^p\frac{q4^{q-p}}{q-p}R^{\frac{\alpha\theta(q-p)}{\theta-1}}+2^q(2^q-1)C_p(R)^p2^p(p+1)^p\mathbf{Var}_{p,\mathcal{E}} (f)^p\\
&\leq 2^{2q+p}(p+1)^p\big(\|f\|_{L^p(X,\mu)}^pR^{\frac{\alpha\theta(q-p)}{\theta-1}}+C_p(R)^p\mathbf{Var}_{p,\mathcal{E}} (f)^p\big).
\end{align*}
Since $\frac{\alpha_p\theta(q-p)}{\theta-1}=\frac{\alpha_p p(q-p)}{q(p/q-1)}=-\alpha_p p$, the latter inequality implies
\begin{multline*}
\|f\|_{L^q(X\,u)}^q\leq 2^{2q+p}(p+1)^p\big(R^{-\alpha_p p}\|f\|_{L^p(X,\mu)}^p+C_p(R)^p\mathbf{Var}_{p,\mathcal{E}} (f)^p\Big) 
\\
\leq  2^{2q+p}(p+1)^pp\big(R^{-\alpha_p}\|f\|_{L^p(X,\mu)}+C_p(R)\mathbf{Var}_{p,\mathcal{E}} (f)\Big)^p.
\end{multline*}
Finally, applying the heat kernel bound $p_{t}(x,y)\leq C_h t^{-\beta}$ to the norm~\eqref{E:Ledoux_norm2}, we get for every $f \ge 0$,
\[
 \| f \|_{L^q(X,\mu)} \le 2^{2+\theta}(p+1)^{\theta}p^{1/q}C_h^{1-\theta}\big(R^{-\alpha_p}\|f\|_{L^p(X,\mu)}+C_p(R)\mathbf{Var}_{p,\mathcal{E}} (f)\Big)^{\theta}\|f\|_{L^1(X,\mu)}^{1-\theta}
\]
for $\beta= \frac{\alpha_p p}{q-p}$, equivalently $\frac{1}{q}=\frac{1}{p}-\frac{\alpha_p }{q\beta}$, as we wanted to prove.
\end{proof}

\subsubsection{Gagliardo-Nirenberg}
In the same lines as~\cite[Section 3.2.7]{SC02}, Theorem~\ref{T:Gagliardo-Nirenberg_loc} extends to the full scale of Gagliardo-Nirenberg inequalities by noticing that for any $t,s>0$ the mapping $f\mapsto(f-t)_+\wedge s:= f_t^s$ is a contraction and hence
\begin{equation}\label{E:local_equiv_a}
R^{-\alpha_p}\|f_t^s\|_{L^p(X,\mu)}+C_p(R) \mathbf{Var}_{p,\mathcal{E}} (f_t^s)\leq C\big(R^{-\alpha_p}\|f\|_{L^p(X,\mu)}+C_p(R) \mathbf{Var}_{p,\mathcal{E}} (f)\big)
\end{equation}
for some constant $C>0$. 
As in the global case, we discuss in the following all these inequalities according to the value of $p \alpha_p$.

\begin{corollary}\label{C:Gagliardo-Nirenberg:corollary1_loc}
Assume that $(\mathrm {PPI}_p(R))$ is satisfied for some $p \ge 1 $ such that $p \alpha_p <\beta$. Then, there exists a constant $C_{p,r,s}>0$ such that for every $f \in W^{1,p}(\mathcal E)$ (or $BV(\mathcal{E})$ for $p=1$),
\begin{equation}\label{E:Gagliardo-Nirenberg:corollary1_loc}
  \| f \|_{L^r(X,\mu)} \le C_{p,r,s}\Big(R^{-\alpha_p}\|f\|_{L^p(X,\mu)}+C_p(R) \mathbf{Var}_{p,\mathcal{E}} (f)\Big)^{\theta}\|f\|_{L^s(X,\mu)}^{1-\theta},
\end{equation}
where $r,s \in [1,+\infty]$ and $\theta \in (0,1]$ are related by the identity
\[
\frac{1}{r}=\theta \Big(\frac{1}{p}-\frac{\alpha_p}{\beta} \Big)+\frac{1-\theta}{s}.
\]
\end{corollary}

\begin{proof}
The proof is the same as in Corollary~\ref{C:Gagliardo-Nirenberg:corollary1} since~\eqref{E:local_equiv_a} corresponds to the property $(H_\infty^+)$ from~\cite[Theorem 3.1]{BCLS95}.
\end{proof}

\begin{remark}
As with the global counterparts, we point out explicitly some particular cases.
\begin{enumerate}[leftmargin=1.5em,label={\rm (\roman*)}]
\item If $r=s$, then $r=\frac{p\beta}{\beta-p \alpha_p}$ and~\eqref{E:Gagliardo-Nirenberg:corollary1_loc} yields the global Sobolev inequality
\[
\| f \|_{L^r(X,\mu)} \le C_{p} \big(R^{-\alpha_p}\|f\|_{L^p(X,\mu)}+C_p(R) \mathbf{Var}_{p,\mathcal{E}} (f)\Big)
\]
\item If $r=p>1$ and $s=1$, then~\eqref{E:Gagliardo-Nirenberg:corollary1_loc} yields the global Nash inequality
\[
\| f \|_{L^p(X,\mu)} \le C_{p} \big(R^{-\alpha_p}\|f\|_{L^p(X,\mu)}+C_p(R) \mathbf{Var}_{p,\mathcal{E}} (f)\Big)^{\theta} \| f \|^{1-\theta}_{L^1(X,\mu)}
\]
with $\theta=\frac{(p-1)\beta}{p(\alpha_p+\beta)-\beta}$.
\item If $s=+\infty$, then~\eqref{E:Gagliardo-Nirenberg:corollary1_loc} yields 
\[
\| f \|_{L^r(X,\mu)} \le C_{p,r} \big(R^{-\alpha_p}\|f\|_{L^p(X,\mu)}+C_p(R) \mathbf{Var}_{p,\mathcal{E}} (f)\Big)^{\theta} \| f \|^{1-\theta}_{L^\infty(X,\mu)}
\]
with $\theta=\frac{p\beta}{r(\beta-p \alpha_p)}$.
\end{enumerate}
\end{remark}

We now turn to the case $p \alpha_p >\beta$.

\begin{corollary}\label{C:Gagliardo-Nirenberg:corollary2_loc}
Assume  that $(\mathrm {PPI}_p(R))$ is satisfied for some $p \ge 1 $ such that $p \alpha_p  > \beta$. Then, there exists a constant $C_p>0$ such that for every $f \in W^{1,p}(\mathcal E)$ (or $BV(\mathcal{E})$ for $p=1$), and $s \ge 1$,
\begin{equation*}
\| f \|_{L^\infty(X,\mu)} \le C_{p} \big(R^{-\alpha_p}\|f\|_{L^p(X,\mu)}+C_p(R) \mathbf{Var}_{p,\mathcal{E}} (f)\Big)^{\theta} \| f \|^{1-\theta}_{L^s(X,\mu)},
\end{equation*}
where $\theta \in (0,1)$ is given by $\theta=\frac{p\beta}{p\beta+s(p\alpha_p-\beta)}$.
\end{corollary}

\begin{proof}
Analogously as Corollary~\ref{C:Gagliardo-Nirenberg:corollary2}, this follows by applying~\cite[Theorem 3.2]{BCLS95} with~\eqref{E:local_equiv_a} and Theorem~\ref{T:Gagliardo-Nirenberg_loc}.
\end{proof}

\subsubsection{Trudinger-Moser}
Trudinger-Moser inequalities correspond to the case $p \alpha_p =\beta$. To treat them, we observe first that Minkowski's inequality together with Lemma~\ref{L:local_norm_approx:a} 
implies
\begin{multline}\label{E:local_equiv_b}
\bigg(\sum_{k\in\mathbb{Z}}\Big(R^{-\alpha_p}\|f_{\rho,k}\|_{L^p(X,\mu)}+C_p(R) \mathbf{Var}_{p,\mathcal{E}} (f_{\rho,k})\Big)^p\bigg)^{1/p}\\
\leq R^{-\alpha_p}\|f\|_{L^p(X,\mu)}+ 2(p+1)C_p(R) \mathbf{Var}_{p,\mathcal{E}} (f)
\end{multline}
for any $p\geq 1$, $\rho>1$ and $f_{\rho,k}:=(f-\rho^k)_+\wedge\rho^k(\rho-1)$.

\begin{corollary}\label{C:Trudinger-Moser_1_loc}
Assume that $(\mathrm {PPI}_1(R))$ is satisfied and that $\alpha_1=\beta$. Then, there exists a constant $C>0$ such that for every $f \in BV(\mathcal{E})$
\begin{align*}
\| f \|_{L^\infty(X,\mu)} \le C \big(R^{-\alpha_p}\|f\|_{L^p(X,\mu)}+C_p(R) \mathbf{Var}_{p,\mathcal{E}} (f)\Big).
\end{align*}
\end{corollary}

\begin{proof}
By virtue of~\eqref{E:local_equiv_b}, the condition $(H_1)$ from~\cite[Section 2]{BCLS95} is satisfied, hence Theorem~\ref{T:Gagliardo-Nirenberg_loc} and~\cite[Theorem 3.2]{BCLS95} yield the result. 
\end{proof}

We finish this section with the Trudinger-Moser inequalities that one obtains for $p>1$.

\begin{corollary}\label{C:Trudinger-Moser_p_loc}
Assume further that $(\mathrm {PPI}(R)_p)$ is satisfied and that $p\alpha_p=\beta$ with $p>1$. Then, there exist constants $c,C>0$ such that 

\begin{align*}
\int_X \left( e^{ c |f|^{\frac{p}{p-1}} }-1 \right) d\mu  \le C \| f \|_{L^1(X,\mu)}
\end{align*}
for every $f \in W^{1,p}(\mathcal{E})$ with 
$\|f\|_{L^p(X,\mu)}=R^{\alpha_p}\big(1-C_p(R)\mathbf{Var}_{p,\mathcal{E}} (f)\big)$. 
\end{corollary}

\begin{proof}
In this case,~\eqref{E:local_equiv_b} implies condition $(H_p)$ from~\cite[Section 2]{BCLS95} for $p>1$, and the result follows from Theorem~\ref{T:Gagliardo-Nirenberg_loc} and~\cite[Theorem 3.4]{BCLS95}.
\end{proof}

\subsection{Examples}

The Gagliardo-Nirenberg and  Trudinger-Moser inequalities proved in this section can be applied in large classes of examples. In particular, we mention the following:
\begin{itemize}
\item \underline{Metric measure spaces with Gaussian heat kernel estimates}: Theorem~\ref{T:BV2} provides the class of strictly local spaces to which one can apply the results obtained in this paper, and in particular Gagliardo-Nirenberg and Trudinger-Moser inequalities. Note that a sufficient condition for condition~\eqref{E:subGauss-upper} to hold is the volume growth condition $\mu (B(x,r)) \ge C r^{d_H}$, in which case one has $\beta = \frac{d_H}{2}$. 
\item \underline{Metric measure spaces with sub-Gaussian heat kernel estimates}: Theorem \ref{T:BV3} yields another large set of examples, including unbounded nested fractals (or product of them). These satisfy $(\mathrm{PPI}_p)$ for $1\le p\le 2$ and condition~\eqref{E:subGauss-upper} with $\beta=\frac{d_H}{d_W}$. In the case of the unbounded Vicsek fractal, its $n$-fold product satisfies $(\mathrm{PPI}_p)$ for $1 \le p \le 2$, c.f. Theorem~\ref{T:Vicsek_prod} and condition~\eqref{E:subGauss-upper} with $\beta=\frac{d_H}{d_W}$. Compact nested fractals satisfy the corresponding localized versions.
\end{itemize}

\section{Morrey's type inequalities}\label{morrey}

The classical Morrey's inequality implies that functions in the Sobolev space $W^{1,p}(\mathbb{R}^d)$ are H\"older continuous (after a possible modification on a set of measure zero) for all $p>d$. This section is devoted to its counterpart in the context of Dirichlet spaces. Besides of being an important inequality on its own, we are interested in the associated critical value 
\[
\delta_\mathcal{E}:=\inf \{ p \ge 1, \, W^{1,p}(\mathcal{E}) \subset C^0(X) \},
\]
where $C^{0}(X)$ denotes the space of a.e bounded functions which admit a continuous representative, and the connection of $\delta_\mathcal{E}$ to other dimensions studied in the metric measure setting~\cite{Kig18}. 

\medskip

The inequality that we prove in this section provides a general embedding of $\mathbf{B}^{p,\alpha}(X)$ into  the  space $C^{\lambda}(X)$, $\lambda>0$, of bounded H\"older functions equipped with the norm
\begin{equation*}
\|f\|_{C^\lambda(X)}:=\|f\|_{L^\infty(X,\mu)}+\mu\text{-ess}\sup\limits_{x\neq y}\frac{|f(x)-f(y)|}{d(x,y)^\lambda}.
\end{equation*}


Those types of embedding, however with weaker regularity, were already observed by Coulhon in~\cite{Cou03} under volume doubling and (sub-)Gaussian heat kernel estimates. Here and throughout this section, we will work under the following additional assumptions:

\begin{itemize}[leftmargin=2em]
\item\underline{Condition 1}. The underlying space is $d_H$-Ahlfors regular;
\item\underline{Condition 2}. The heat semigroup admits a heat kernel with Gaussian or sub-Gaussian estimates. 
\end{itemize}

\subsection{Metric approach}

 The proof of the following result is based on a generalization of the ideas in~\cite[Theorem 8.1]{Gri03}. 
 Notice that Theorem~\ref{T:Morrey_local} holds for any pair of exponents $(p,\alpha)$; Morrey's inequality will correspond to the specific pairs $(p,\alpha_p)$.

\begin{theorem}\label{T:Morrey_local}
For any $p>\frac{d_H}{d_W\alpha}$ and $R>0$, there exists $C_p>0$ (independent from $R$) 
such that
\begin{equation}\label{E:Morrey_local}
\mu\text{-ess}\sup\limits_{0<d(x,y)<R/3}\frac{|f(x)-f(y)|}{d(x,y)^\lambda} \leq C\|f\|_{p,\alpha,R}
\end{equation}
for any $f\in\mathbf{B}^{p,\alpha}(X)$, where $\lambda=d_W\alpha-\frac{d_H}{p}$. In particular, if $\alpha p > \frac{d_H}{d_W}$, then $\mathbf{B}^{p,\alpha}(X) \subset C^\lambda (X)$, where $\lambda=d_W\alpha-\frac{d_H}{p}$.
\end{theorem}

\begin{remark}
We note that when applied to the critical exponent $\alpha=\alpha_p$ the condition $\alpha_p p = \frac{d_H}{d_W}  $ exactly corresponds to the critical exponent for Trudinger-Moser inequalities in the previous section.
\end{remark}

\begin{proof}
Let first $0<r<R/3$ and consider $x,y\in X$ with $d(x,y)\leq r$. Define
\begin{equation*}
f_r(x):=\frac{1}{\mu\big(B(x,r)\big)}\int_{B(x,r)}u(z)\,d\mu(z)
\end{equation*}
and notice that
\[
f_r(x)=\frac{1}{\mu\big(B(x,r)\big)\mu\big(B(y,r)\big)} \int_{B(x,r)}\int_{B(y,r)}u(z)\,d\mu(z')\,d\mu(z).
\]
Analogously one defines $f_r(y)$. H\"older's inequality yields
\begin{align*}
|f_r(x)-f_r(y)|&=\frac{1}{\mu\big(B(x,r)\big)\mu\big(B(y,r)\big)}\Big|\int_{B(x,r)}\int_{B(y,r)}(u(z)-u(z'))\,d\mu(z')\,d\mu(z)\Big|\\
&\leq\bigg(\frac{1}{\mu\big(B(x,r)\big)\mu\big(B(y,r)\big)}\int_{B(x,r)}\int_{B(y,r)}|u(z)-u(z')|^p\,d\mu(z')\,d\mu(z)\Big|\bigg)^{1/p}
\end{align*}
hence, applying the $d_H$-Ahlfors regularity of the space and since $d(x,y)\leq r$,
 we get
\begin{align*}
|f_r(x)-f_r(y)|^p&\leq \frac{C}{r^{2d_H}}\int_X\int_{B(z,3r)}|u(z)-u(z')|^p\,d\mu(z')\,d\mu(z)\\
&\leq Cr^{p\alpha d_W-d_H}\sup_{r\in(0,R/3)}\frac{1}{r^{d_H+p\alpha d_W}}\int_X\int_{B(z,3r)}|u(z)-u(z')|^p\,d\mu(z')\,d\mu(z)\\
&\leq Cr^{p\alpha d_W-d_H}\|f\|_{p,\alpha,R}^p,
\end{align*}
where the last inequality follows from the characterization of $\mathbf{B}^{p,\alpha}(X)$ as a Korevaar-Schoen class space, see e.g.~\eqref{E:BV2_02} for the Gaussian and~\cite[Theorem 2.4]{ABCRST3} for the sub-Gaussian case. Thus,
\begin{equation*}
|f_r(x)-f_r(y)| \leq C^{1/p}r^{\alpha d_W-\frac{d_H}{p}}\|f\|_{p,\alpha,R}
\end{equation*}
and analogously one obtains
\begin{equation}\label{E:Morrey_local_02}
|f_{2r}(x)-f_r(x)| \leq \tilde{C}^{1/p}r^{\alpha d_W-\frac{d_H}{p}}\|f\|_{p,\alpha,R}.
\end{equation}
Let now $x\in X$ be a Lebesgue point of $f$. Setting $r_k=2^{-k}r$, $k=0,1,2\ldots$, the latter inequality yields
\begin{equation}\label{E:Morrey_local_03}
|f(x)-f_r(x)|\leq\sum_{k=0}^\infty|f_{r_k}(x)-f_{r_{k+1}}(x)|\leq 
\tilde{C}^{1/p}r^{\alpha d_W-\frac{d_H}{p}}\|f\|_{p,\alpha,R}.
\end{equation}
Let $y\in X$ be another Lebesgue point of $f$ such that $d(x,y)\leq R/3$. Applying the triangle inequality as well as~\eqref{E:Morrey_local_02} and~\eqref{E:Morrey_local_03} with $r=d(x,y)$ we obtain
\begin{multline}\label{E:Morrey_local_04}
|f(x)-f(y)|\leq |f(x)-f_r(x)|+|f_r(x)-f_r(y)|+|f_r(y)-f(y)|\\
\leq C_{p}d(x,y)^{\alpha d_W-\frac{d_H}{p}}\|f\|_{p,\alpha,R}.
\end{multline}
Then, by virtue of~\cite[Theorem 3.4.3]{HKST15}, the volume doubling property of the space implies the validity of the Lebesgue differentiation theorem from~\cite[Section 3.4]{HKST15}, which guarantees that the set of Lebesgue points of $f$ is dense  in $X$. Thus,~\eqref{E:Morrey_local_04} implies~\eqref{E:Morrey_local}.
Finally, for any fixed $r>0$ (e.g. $r=R/4$), H\"older's inequality yields $|f_r(x)|\leq r^{-\frac{d_H}{p}}\|f\|_{L^p(X,\mu)}$, which together with~\eqref{E:Morrey_local_03} implies
\[
|f(x)|\leq C_r(\|f\|_{L^p(X,\mu)}+\|f\|_{p,\alpha,R})
\]
$\mu$-a.e. $x\in X$. Thus, $L^\infty(X,\mu)\subseteq\mathbf{B}^{p,\alpha}(X)$.
\end{proof}

Since the constant $C_p$ in the previous theorem is independent of $R$, by letting $R \to +\infty$ one deduces the corresponding global inequality.
\begin{corollary}\label{T:Morrey_global}
For any $p>\frac{d_H}{d_W\alpha}$, there exists $C_p>0$ such that
\begin{equation*}
\mu\text{-ess}\sup\limits_{d(x,y)>0}\frac{|f(x)-f(y)|}{d(x,y)^\lambda} \leq C_p\|f\|_{p,\alpha}
\end{equation*}
for any $f\in\mathbf{B}^{p,\alpha}(X)$, where $\lambda=d_W\alpha-\frac{d_H}{p}$.
\end{corollary}

\subsection{Heat semigroup approach}

A drawback of Theorem~\ref{T:Morrey_local} is that when applied to the pair $(p,\alpha_p)$, it would be sharper and more natural to get on the right hand side of~\eqref{E:Morrey_local} the $p$-variation $\mathbf{Var}_{p,\mathcal{E}}(f)$ instead of the Besov semi-norm $\| \cdot \|_{p,\alpha_p,R}$. This certainly requires more assumptions than just sub-Gaussian heat kernel estimates and Ahlfors regularity. So, in addition to 
the latter,  
we will also assume
in this section 
the weak Bakry-\'Emery type estimate $(\mathrm G_\infty)$ from~\eqref{weakBE1}.

\begin{itemize}[leftmargin=2em]
\item\underline{Condition 3}. There exists a constant $C>0$ so that for any $f \in L^\infty(X,\mu)$, and $x,y \in X$
\begin{equation*}
| P_t f (x)-P_t f(y)| \le C \frac{d(x,y)^{d_W(1-\alpha_1)}}{t^{1-\alpha_1}} \| f \|_{L^\infty(X,\mu)}
\end{equation*}
for all $t >0$. 
\end{itemize}
Under these assumptions, we start by presenting the key estimate to obtain an \textit{almost} optimal Morrey's type inequality. Its proof relies on some ideas first developed by T. Coulhon~\cite{Cou03} and E.M. Ouhabaz~\cite{Ouh98}. In the sequel, $\Delta$ will denote the infinitesimal generator of the Dirichlet form $(\mathcal{E},\mathcal{F})$.

\begin{theorem}\label{Morrey-BesselBesov}
Let $p>1$ and $\frac{d_H}{pd_W} < \alpha < \frac{d_H}{pd_W}+\big(1-\frac{1}{p}\big)(1-\alpha_1)$. Then,
\[
| f(x)-f(y)| \le  Cd(x,y)^{\alpha d_W-\frac{d_H}{p}} \| (-\Delta)^\alpha f \|_{L^p(X,\mu)}
\]
for $ f\in {\rm dom}\, (-\Delta)^\alpha$, and $\mu$-a.e. $x,y \in X$.
\end{theorem}

We decompose the proof into several lemmas; the first is a direct consequence of the heat kernel upper bound, and the second uses the fact that $(\mathrm G_\infty)$ is equivalent to 
\[
|p_t(x,z)-p_t(y,z)|\leq C\frac{d(x,y)^{d_W(1-\alpha_1)}}{t^{1-\alpha_1+\frac{d_H}{d_W}}}
\]
for some $C>0$ and every $t>0$, $x,y,z\in X$, see~\cite[Lemma 3.4]{ABCRST3}.

\begin{lemma}\label{L:Pt_sup_bound}
Let $p \ge 1$. There exists a constant $C>0$ such that for every $f \in L^p(X,\mu)$, $t > 0$ and $\mu$ a.e. $x \in X$,
\[
| P_t f(x) | \le  \frac{C}{t^{\frac{d_H}{pd_W}}} \| f \|_{L^p(X,\mu)}.
\]
\end{lemma}
%
\begin{lemma}\label{L:Pt_pHolder}
Let $p \ge 1$. There exists a constant $C>0$ such that for every $f \in L^p(X,\mu)$, $t > 0$ and $\mu$ a.e. $x,y \in X$,
\[
| P_t f(x) -P_tf(y) | \le C \frac{d(x,y)^{d_W(1-\alpha_1) \big(1-\frac{1}{p}\big)}}{t^{\frac{d_H}{pd_W}+(1-\alpha_1)\big(1-\frac{1}{p}\big)}} \| f \|_{L^p(X,\mu)}.
\]
\end{lemma}
%

The third lemma is more involved and we provide its proof.
\begin{lemma}\label{L:pre-Morrey-Bessel}
Let $\frac{d_H}{pd_W} < \alpha < \frac{d_H}{pd_W}+\big(1-\frac{1}{p}\big)(1-\alpha_1)$. There exists a constant $C>0$ such that for every $f \in L^2(X,\mu)$ and $\mu$-a.e. $x,y \in X$,
\[
\int_0^{+\infty} t^{\alpha-1} | P_t f (x) -P_tf(y)| dt \le Cd(x,y)^{\alpha d_W-\frac{d_H}{p}} \| f \|_{L^p(X,\mu)}.
\]
\end{lemma}

\begin{proof}
The idea is to split the integral into two parts,
\[
\int_0^{+\infty} t^{\alpha-1} | P_t f (x) -P_tf(y)| dt =\int_0^{\delta} t^{\alpha-1} | P_t f (x) -P_tf(y)| dt+\int_\delta^{+\infty} t^{\alpha-1} | P_t f (x) -P_tf(y)| dt,
\]
where $\delta >0$ will be chosen later. First, by Lemma~\ref{L:Pt_sup_bound} we have
\begin{align*}
\int_0^{\delta} t^{\alpha-1} | P_t f (x) -P_tf(y)| dt &\le \int_0^{\delta} t^{\alpha-1}( | P_t f (x)| +|P_tf(y)| )dt \\
&\le \int_0^{\delta} t^{\alpha-1} \frac{C}{t^{\frac{d_H}{pd_W}}} dt \| f \|_{L^p(X,\mu)} 
\le C \delta^{\alpha- \frac{d_H}{pd_W}} \| f \|_{L^p(X,\mu)}.
\end{align*}
As usual, the constant $C$ in the previous inequalities may change from line to line. 
Secondly, applying Lemma~\ref{L:Pt_pHolder} we get
\begin{align*}
\int_\delta^{+\infty} t^{\alpha-1}  | P_t f (x) -P_tf(y)| dt &\le C \int_\delta^{+\infty} t^{\alpha-1} \frac{d(x,y)^{d_W(1-\alpha_1)\big(1-\frac{1}{p}\big)}}{t^{\frac{d_H}{pd_W}+(1-\alpha_1)\big(1-\frac{1}{p}\big)}} \| f \|_{L^p(X,\mu)}dt \\
 & \le  C d(x,y)^{d_W(1-\alpha_1)\big(1-\frac{1}{p}\big)}\!\! \int_\delta^{+\infty} t^{\alpha-1- \frac{d_H}{pd_W}-(1-\alpha_1)\big(1-\frac{1}{p}\big)} dt \| f \|_{L^p(X,\mu)} \\
 &\le C  d(x,y)^{d_W(1-\alpha_1)\big(1-\frac{1}{p}\big)} \delta^{\alpha- \frac{d_H}{pd_W}-(1-\alpha_1)\big(1-\frac{1}{p}\big)} \| f \|_{L^p(X,\mu)}.
\end{align*}
Thus, one concludes
\begin{multline*}
\int_0^{+\infty} t^{\alpha-1} | P_t f (x) -P_tf(y)| dt\\
\le C \Big(\delta^{\alpha- \frac{d_H}{pd_W}}  + d(x,y)^{d_W(1-\alpha_1)\big(1-\frac{1}{p}\big)} \delta^{\alpha- \frac{d_H}{pd_W}-(1-\alpha_1)\big(1-\frac{1}{p}\big)}  \Big)\| f \|_{L^p(X,\mu)}
\end{multline*}
and choosing $\delta=d(x,y)^{d_W}$ yields the result.
\end{proof}

We are finally ready to prove Theorem~\ref{Morrey-BesselBesov}.

\begin{proof}[Proof of Theorem~\ref{Morrey-BesselBesov}]
Let $f \in {\rm dom}\,(-\Delta)^{-\alpha}$. By virtue of Lemma~\ref{L:pre-Morrey-Bessel},
\begin{multline*}
| (-\Delta)^{-\alpha} f (x) - (-\Delta)^{-\alpha} f (y) |= C \left|  \int_0^{+\infty} t^{\alpha-1} (P_tf(x)-P_tf(y))  dt \right| \\
 \le C  \int_0^{+\infty} t^{\alpha-1}  |P_tf(x) -P_tf(y) |\, dt 
 \le  Cd(x,y)^{\alpha d_W-\frac{d_H}{p}} \| f \|_{L^p(X,\mu)}.
\end{multline*}
Applying the inequality to $(-\Delta)^{\alpha}f$ instead of $f$ yields the result.
\end{proof}

As a consequence, we deduce a version of a Morrey's type inequality which is \textit{almost} optimal. In addition to Ahlfors regularity, sub-Gaussian heat kernel estimates and condition $(\mathrm G_\infty)$, it will be necessary to assume the property $({\rm PPI}_p)$.

\begin{theorem}\label{Morrey-BesselBesov2}
Let $p>1$ and $\frac{d_H}{pd_W} < \alpha_p < \frac{d_H}{pd_W}+\big(1-\frac{1}{p}\big)(1-\alpha_1)$. Assuming $(\mathrm G_\infty)$ and $({\rm PPI}_p)$, for every $0<\alpha <\alpha_p$ there exists a constant $C>0$ such that 
\[
| f(x)-f(y)| \le  Cd(x,y)^{\alpha d_W-\frac{d_H}{p}} \|  f \|^{1-\frac{\alpha}{\alpha_p}}_{L^p(X,\mu)} \mathbf{Var}_{p,\mathcal{E}}(f)^{\frac{\alpha}{\alpha_p}}
\]
for every $f \in W^{1,p}(\mathcal E)$ and $\mu$-a.e. $x,y \in X$.
\end{theorem}

\begin{proof}
Let $f \in W^{1,p}(\mathcal E)$. For $\delta >0$, applying $({\rm PPI}_p)$ one has
\begin{align*}
 \left\| \int_0^\infty t^{-s-1} (P_t f - f)\ dt \right\|_{L^p(X,\mu)} & \le  \int_0^\infty t^{-s-1}  \| P_t f - f \|_{L^p(X,\mu)}  dt \\
  & \le \int_0^\delta t^{-s-1}  \| P_t f - f \|_{L^p(X,\mu)}  dt +  \int_\delta^\infty t^{-s-1}  \| P_t f - f \|_{L^p(X,\mu)}  dt  \\
  & \le \mathbf{Var}_{p,\mathcal{E}}(f) \int_0^\delta t^{-s-1 +\alpha_p}   dt +2 \| f \|_{L^p(X,\mu)} \int_\delta^\infty t^{-s-1}    dt \\
  & \le \mathbf{Var}_{p,\mathcal{E}}(f) \frac{\delta^{\alpha-s}}{\alpha_p -s}+2 \| f \|_{L^p(X,\mu)} \frac{\delta^{-s}}{s}.
\end{align*}
Finally, since
\[
\| (-\Delta)^\alpha f \|_{L^p(X,\mu)}=C\left\| \int_0^\infty t^{-\alpha-1} (P_t f - f)\ dt \right\|_{L^p(X,\mu)},
\]
the result follows from Theorem~\ref{Morrey-BesselBesov} by optimizing in $\delta$.
\end{proof}

\subsection{Examples}

As an illustration of the more concrete regularity results that can be obtained from the Morrey's inequality in Theorem~\ref{T:Morrey_local}, in this paragraph we apply that result to several settings covered by the general theory. In addition, we propose new conjectures for fractals in the case $p>1$. As we already mentioned, Morrey's inequality is specially interesting at the critical exponent $\alpha_p$, since it provides the (H\"older) regularity of the functions in the Sobolev space $W^{1,p}(\mathcal E)$.
Recall that we define the Sobolev continuity exponent of a Dirichlet form as
\[
\delta_\mathcal{E}=\inf \{ p \ge 1, \, W^{1,p}(\mathcal{E}) \subset C^0(X) \}.
\]

\subsubsection*{Strictly local Dirichlet spaces}
In the framework described in Section~\ref{sec:mmsdoubling}, we know from Theorem~\ref{T:BV2}(ii) that under the quasi Bakry-\'Emery condition~\eqref{E:sBE_BV2}, the local Besov semi-norm $\|f\|_{\alpha_p,p,R}$ is equivalent to the $L^p$-norm of the gradient and $\alpha_p=1/2$ for any $p\geq 2$. Hence, Theorem~\ref{T:Morrey_local} recovers the classical Morrey inequality.

\begin{theorem}\label{T:continuity_BV2}
Let $(X,d,\mu)$ be a metric measure space that satisfies the volume doubling property and supports a 2-Poincar\'e inequality. Moreover, assume that it satisfies the quasi Bakry-\'Emery condition~\eqref{E:sBE_BV2}. Then, for any $p>d_H$, there exists $C>0$ such that
\begin{equation*}
\sup_{0< d(x,y) \le R }\frac{|f(x)-f(y)|}{d(x,y)^{1-\frac{d_H}{p}}}\leq C \||\nabla f|\|_{L^p(X,\mu)}.
\end{equation*}
In particular $\delta_\mathcal{E} \le d_H$.
\end{theorem}
%
%
\subsubsection*{Nested fractals}
Currently, dealing with strongly local Dirichlet spaces with sub-Gaussian heat kernel estimates is more delicate due to the lack of an analogue to the quasi Bakry-\'Emery condition~\eqref{E:sBE_BV2}. Nevertheless, we would like to discuss several conjectures for nested fractals and the Sierpinski carpet that arise in the light of those presented in~\cite{ABCRST3}. 
In view of recent developments, specially in the fractal setting~\cite[Section 19]{Kig18}, it seems that the exponent $\delta_{\mathcal{E}}$ may be related to the so-called Ahlfors regular conformal dimension of the space. We leave this question open for possible future research.

\medskip

We start with the case of the Vicsek set discussed in Section~\ref{subS:fractals}, which is our best understood fractal model so far. In the next theorem, $X$ thus denotes this particular set.

\begin{theorem}\label{T:Vicsek continuity}
For the Vicsek set, $\delta_\mathcal{E}=1$. Moreover, $W^{1,p}(\mathcal E) \subset C^{1-1/p}(X)$ for any $p>1$.
\end{theorem}

\begin{proof}
The condition for the possible ranges of $p$ is obtained as follows.  Recall from Theorem~\ref{T:Morrey_local} that we look for the infimum of the $p$'s such that $\frac{d_H}{p}<d_W\alpha_p$. For Vicsek set, we know from Theorem \ref{T:Vicsek1} and~\cite[Theorem 3.11]{ABCRST3}  that we always have 
{\small
\[
d_W\alpha_p \ge d_W\Big(1-\frac{d_H}{d_W}\Big)\Big(1-\frac{2}{p}\Big)+\frac{d_W}{p}=(d_W-d_H)\Big(1-\frac{2}{p}\Big)+\frac{d_W}{p}=\frac{(d_W-d_H)(p-2)+d_W}{p}.
\]
}
Thus, the condition for $p$ becomes
\[
d_H<(d_W-d_H)(p-2)+d_W
\]
which is equivalent to 
$p>1$. Theorem~\ref{T:Morrey_local} also yields $W^{1,p}(\mathcal E) \subset C^\lambda (X)$ with
\[
\lambda=d_W\alpha_p-\frac{d_H}{p} \ge (d_W-d_H)\Big(1-\frac{1}{p}\Big) =1-\frac{1}{p},
\]
where the last equality follows from the fact that on the Vicsek set $d_W-d_H=1$.
\end{proof}

For a generic nested fractal $X$ we can provide bounds for the critical exponent $\delta_\mathcal{E}$.

\begin{theorem}\label{T:continuity_nested}
On nested fractals, $1 \le \delta_\mathcal{E} \le \frac{2 d_H}{d_W}$. Moreover, $W^{1,p}(\mathcal E) \subset C^{ \lambda }(X)$ for any $p\ge 2$ with
\begin{equation*}
\lambda=(d_W-d_H)\Big(1-\frac{1}{p}\Big).
\end{equation*}
\end{theorem}

\begin{proof}
From~\cite[Theorem 3.11]{ABCRST3}, we know that $\alpha_p \ge \frac{1}{2}$ for $1 \le p \le 2$ and 
\[
\alpha_p \ge \Big(1-\frac{d_H}{d_W}\Big)\Big(1-\frac{2}{p}\Big)+\frac{d_W}{p}
\]
for $p \ge 2$. 
The result now follows as in the proof of Theorem~\ref{T:Vicsek continuity}.
\end{proof}

%

%

Since it is conjectured in~\cite[Section 5]{ABCRST3} that on all nested fractals one has for every $p \ge 1$,
\begin{equation*}
\alpha_p=\Big(1-\frac{d_H}{d_W}\Big)\Big(1-\frac{2}{p}\Big)+\frac{1}{p},
\end{equation*}
we can actually state the following more precise conjecture.

\begin{conjecture}
On nested fractals, $\delta_\mathcal{E}=1$ and for any $p>1$, there exists $C>0$ such that such that
\begin{equation*}
\mu\text{-ess}\sup\limits_{x\neq y}\frac{|f(x)-f(y)|}{d(x,y)^\lambda} \leq C \mathbf{Var}_{p,\mathcal{E}}(f)
\end{equation*}
 for every $f \in W^{1,p}(\mathcal E)$ with
\begin{equation*}
\lambda=(d_W-d_H)\Big(1-\frac{1}{p}\Big).
\end{equation*}
In particular for the Sierpinski gasket, 
$\lambda=\frac{\log(5/3)}{\log 2}\big(1-\frac{1}{p}\big)$ and for the Vicsek set, $\lambda=1-\frac{1}{p}$.
\end{conjecture}

The Sierpinski carpet is of different nature and it has been conjectured in~\cite[Conjecture 5.4]{ABCRST3} that $\alpha_1=(d_H-d_{tH}+1)/d_W$ and
\[
\alpha_p = \Big(1-\frac{2}{p} \Big) (1-\alpha_1)+\frac{1}{p}
\]
for $p >1$, where $d_{tH}$ is the topological Hausdorff dimension of the carpet. After some elementary computations, this yields the following conjecture.

\begin{conjecture}
For the Sierpinski carpet, $\delta_\mathcal{E}=2-\frac{d_W-d_H}{d_W-d_H+d_{tH}-1}$ and for any $p>\delta_\mathcal{E}$, there exists $C>0$ such that 
\begin{equation*}
\mu\text{-ess}\sup\limits_{x\neq y}\frac{|f(x)-f(y)|}{d(x,y)^\lambda} \leq C \mathbf{Var}_{p,\mathcal{E}}(f)
\end{equation*}
for every $f \in W^{1,p}(\mathcal E)$ with
\begin{equation*}
\lambda=\frac{(d_W-d_H+d_{tH}-1)(p-2)+d_W}{p}-\frac{d_H}{p}.
\end{equation*}
\end{conjecture}
Since for the Sierpinski carpet it is known that $d_H=\frac{\log 8}{\log 3}=\frac{3\log 2}{\log 3}$ and $d_{tH}=1+\frac{\log 2}{\log 3}$, $d_W\approx 2.097$, this gives $d_W-d_H+d_{tH}-1=d_W-\frac{2\log 2}{\log 3}$. The critical exponents thus read
\[
\delta_\mathcal{E}=1+\frac{\log 2}{d_W\log 3-2\log 2}
\]
and
\[
\lambda=\frac{(d_W\log 3-2\log 2)(p-2)+d_W\log 3-3\log 2}{p\log 3}
=d_W\Big(1-\frac{1}{p}\Big)-\frac{\log 2}{\log 3}\Big(2-\frac{1}{p}\Big).
\]

\subsection*{Acknowledgments}
The first author thanks A. Teplyaev for his valuable input in the discussion of the Vicsek set.
\bibliographystyle{amsplain}
\bibliography{Gagliardo_Refs}

\noindent
P. Alonso Ruiz: \texttt{paruiz@math.tamu.edu}  \\
Department of Mathematics,
Texas A{\&}M University,
College Station, TX 77843

\

\noindent

F. Baudoin: \texttt{fabrice.baudoin@uconn.edu}
\\
Department of Mathematics,
University of Connecticut,
Storrs, CT 06269

\end{document}